\documentclass[12pt, reqno]{article}

\usepackage{amsmath,amssymb,amsfonts,amsthm,graphicx}
\usepackage{setspace} 
\usepackage{color}
\usepackage{hyperref}
\usepackage{mathrsfs}
\hypersetup{colorlinks=true, linkcolor=blue, 
citecolor=red, urlcolor=blue, pdfborder={0 0 0}}            

\newtheorem{thm}{Theorem} 
\newtheorem{prop}{Proposition} 
\newtheorem{lemma}{Lemma}
\newtheorem{corollary}{Corollary} 
\theoremstyle{definition}
\newtheorem{defn}{Definition}
\newtheorem{remark}{Remark}

\renewcommand{\P}{\mathbb P}
\newcommand{\Z}{\mathbb Z}
\newcommand{\E}{\mathbb E}
\newcommand{\N}{\mathbb N}

\newcommand{\dgr}{d_{\textup{gr}}}
\newcommand{\gcrit}{g_{3,\textup{cr}}}
\newcommand{\UIHPQ}{\textup{\textsf{UIHPQ}}}
\newcommand{\UIHPT}{\textup{\textsf{UIHPT}}}
\newcommand{\UIPQ}{\textup{\textsf{UIPQ}}}

\newcommand{\dmap}{d_{\textup{map}}}

\newcommand{\br}{\textup{\textsf{b}}}

\newcommand{\LT}{\textup{\textsf{LT}}}
\newcommand{\DS}{\textup{\textsf{DS}}}
\newcommand{\LS}{\textup{\textsf{LS}}}
\newcommand{\TB}{\textup{\textsf{TB}}}
\newcommand{\m}{\textup{$\mathfrak{m}$}}

\newcommand{\gmax}{\gamma_{\textup{max}}}
\newcommand{\gmin}{\gamma_{\textup{min}}}

\newcommand{\tgmax}{\tilde{\gamma}_{\textup{max}}}
\newcommand{\tgmin}{\tilde{\gamma}_{\textup{min}}}
\newcommand{\tcR}{\tilde{\mathcal{R}}}
\newcommand{\tg}{\tilde{g}}
\newcommand{\tf}{\tilde{f}}
\newcommand{\tDelta}{\tilde{\Delta}}
\newcommand{\tbr}{\tilde{\textup{\textsf{b}}}}
\newcommand{\tT}{\tilde{T}}

\newcommand{\cQ}{\mathcal{Q}}
\newcommand{\cR}{\mathcal{R}}

\newcommand{\cb}{\textup{Ball}}
\newcommand{\suc}{\textup{succ}}

\begin{document}

\title{Geodesic rays in the uniform infinite half-planar quadrangulation
  return to the boundary}
\author{Erich Baur\footnote{erich.baur@ens-lyon.fr},\quad Gr\'{e}gory
  Miermont\footnote{gregory.miermont@ens-lyon.fr},\quad Lo\"ic
  Richier\footnote{loic.richier@ens-lyon.fr}\\ ENS Lyon} 
\date{\small \today}
\maketitle
\thispagestyle{empty}

\begin{abstract} 
  We show that all geodesic rays in the uniform infinite half-planar
  quadrangulation ($\UIHPQ$) intersect the boundary infinitely many times,
  answering thereby a recent question of Curien. However, the possible
  intersection points are sparsely distributed along
  the boundary. As an intermediate step, we show that geodesic rays in the
  $\UIHPQ$ are proper, a fact that was recently established by Caraceni and
  Curien in~\cite{CaCu} by a reasoning different from ours. Finally, we
  argue that geodesic rays in the uniform infinite half-planar
  triangulation behave in a very similar manner, even in a strong
  quantitative sense.
\end{abstract} 

\footnote{{\it Acknowledgment of support.} The research of EB
  was supported by the Swiss National Science Foundation grant
  P300P2\_161011, and performed within the framework of the LABEX MILYON
  (ANR-10-LABX-0070) of Universit\'e de Lyon, within the program
  ``Investissements d'Avenir'' (ANR-11-IDEX-0007) operated by the French
  National Research Agency (ANR). GM is a member of Institut
  Universitaire de France, and acknowledges support of the grant
  ANR-14-CE25-0014 (GRAAL) and of Fondation Simone et Cino Del Duca. }

\section{Introduction}
The uniform infinite half-planar quadrangulation $\UIHPQ$ provides a
natural model of (discrete) random half-planar geometry. It arises as a
local limit of finite-size quadrangulations with a boundary, when the
number of quadrangles and the size of the boundary tend to infinity in a
suitable way. We give more precise statements with references in the next
section.

The full-plane equivalent of the $\UIHPQ$ is the so-called uniform infinite
planar quadrangulation ($\UIPQ$), which was introduced by Krikun~\cite{Kr},
after Angel and Schramm's pioneering work on triangulations~\cite{AnSc}.
It is proved in~\cite{CuMeMi} that geodesic rays (i.e., infinite one-ended
geodesics) starting from the root in the $\UIPQ$ satisfy a confluence 
property towards infinity (and, as it is also shown, towards the root):
Almost surely, there exists an infinite set of vertices such that every
geodesic ray emanating from the origin passes through all the vertices of
this set. In other words, geodesic rays in the $\UIPQ$ are essentially
unique, in the sense that the Gromov boundary of the $\UIPQ$ contains only
a single point.

In a recent work~\cite{CaCu}, Caraceni and Curien showed that the analog
coalescence property of geodesics holds in the half-planar model $\UIHPQ$:
There is with probability one an infinite sequence of distinct vertices,
which are all hit by every geodesic ray emanating from the root. Our main
result of this paper shows that this property holds in the $\UIHPQ$ in a
very strong sense.
\begin{thm}
\label{thm:geod4}
Almost surely, every geodesic ray in the $\UIHPQ$ hits the boundary
infinitely many times. More specifically, almost surely there is an
infinite sequence of distinct vertices all lying on the boundary of the
$\UIHPQ$, such that every geodesic ray passes through every point of this
sequence except maybe for a finite number.
\end{thm}

\begin{figure}[ht]
	\begin{center}
		\includegraphics[scale=0.8]{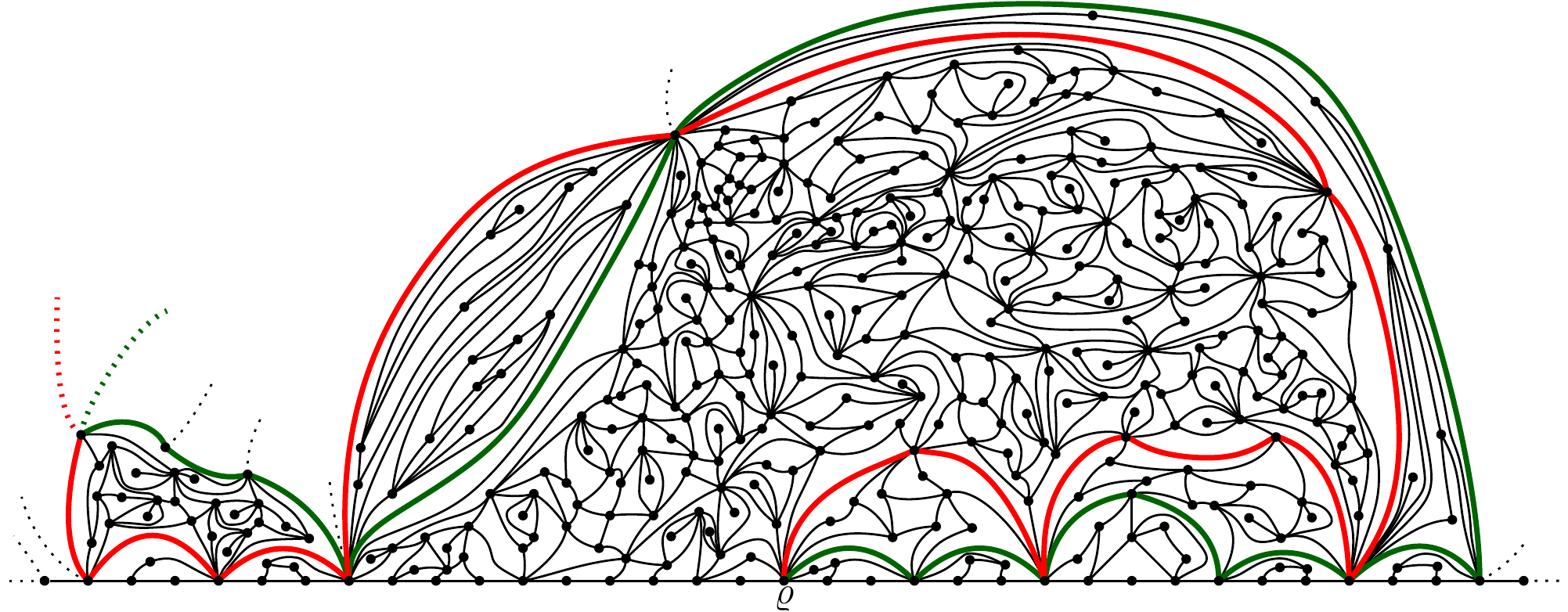}
	\end{center}
	\caption{Artistic drawing of the $\UIHPQ$ (here, for simplicity,
          with a simple boundary) with two distinguished geodesics emanating
          from the root vertex $\rho$ called the {\it maximal} or {\it
            leftmost geodesic} (in red) and the {\it minimal} or {\it
            rightmost geodesic} (in green). All geodesic rays starting from
          $\rho$ lie in between the maximal and minimal geodesic. Their
          joint intersection points with the boundary are thus
          intersection points for {\it any} geodesic ray emanating from
          $\rho$.}
	\label{fig:UIHPQ-MinMaxGeod}
\end{figure}

After having introduced some notation, we will outline our strategy for
proving Theorem~\ref{thm:geod4} at the beginning of
Section~\ref{sec:mainproofs}.  In Section~\ref{sec:sparseness}, we obtain
more precise information on the set of times (and points) of intersection
with the boundary, see Proposition~\ref{prop:geod4sparse}. More
specifically, by analyzing two distinguished geodesics starting from the
root vertex, we will construct an infinite set of boundary vertices, which
contains all possible points of intersection with any geodesic ray. See
Figure~\ref{fig:UIHPQ-MinMaxGeod}. Our
construction will imply that geodesic rays hit both ``sides'' of the
boundary (see Section~\ref{sec:BoundaryMap} for the exact terminology)
infinitely many times; however, the time between two hits has a logarithmic
tail. Section~\ref{sec:UIHPT} contains an extension of our results to the
uniform infinite half-planar triangulation $\UIHPT$, see
Theorem~\ref{thm:geod3}.

The $\UIHPQ$ considered here has a non-simple boundary, meaning that the
boundary vertices cannot be connected by a simple curve. In other words,
there are pinch-points along the boundary. The analog of the $\UIHPQ$ with
a simple boundary, which we denote by $\UIHPQ^{(s)}$ (see~\cite{AnCu,CuMi},
and~\cite{An} for the triangular analog), can be constructed by a pruning
procedure applied to the $\UIHPQ$, cf.~\cite{CuMi}, and this construction
will allow us to argue in Corollary~\ref{cor:UIHPQ-simple} that our results
on geodesics transfer to the $\UIHPQ^{(s)}$.

The uniform infinite planar quadrangulation $\UIPQ$ contains a
distinguished infinite sequence of vertices, the so-called {\it
  spine}. This sequence can be interpreted as a self-avoiding infinite path
in the $\UIPQ$, which is, as it is shown in~\cite{CuMeMi}, almost surely
hit only a {\it finite} number of times by the collection of geodesic rays
starting from the root. This result should be seen in comparison with our
Theorem~\ref{thm:geod4}, see Remark~\ref{rem:intersection-spine} for more
on this. In particular, in the $\UIPQ$, there are self-avoiding paths of
infinite length which are finally avoided by any geodesic ray. As our
arguments leading to Theorem~\ref{thm:geod4} show, such paths do not exist
in the $\UIHPQ$: Any infinite self-avoiding path in the $\UIHPQ$ must cross
any geodesic ray infinitely often.

The fact that the spine is eventually left by the collection of geodesic
rays emanating from the root is a key step in~\cite{CuMeMi} to prove the
confluence property towards infinity, and our approach borrows to some
extent from the ideas presented there.

We will rely on a Schaeffer-type encoding of the $\UIHPQ$ going back
to~\cite{Sc,BoDFGu,CuMi} in terms of uniformly labeled critical Galton-Watson
trees, which are attached to the down-steps of a two-sided simple random
walk. The key observation for Theorem~\ref{thm:geod4} is expressed in
Proposition~\ref{prop:Delta_0}. There, we find the exact distribution of 
the minimal label, which is attained in the trees attached to an excursion
above $-1$ of the simple random walk. A related quantity is studied in
Lemma 14 in~\cite{CuMeMi}, see also Remark~\ref{rem:intersection-spine}
below.  In the last section, we argue that a variant of the Schaeffer-type
encoding can be used to construct the uniform infinite half-planar
triangulation $\UIHPT$, and then a similar strategy works for the $\UIHPT$
as well, resulting in Theorem~\ref{thm:geod3}. In particular, somewhat
surprisingly, we will see that geodesic rays in the $\UIHPT$ behave in a
quantitatively very similar manner.

\section{The uniform infinite half-planar quadrangulation}
The $\UIHPQ$ is an infinite random quadrangulation with an infinite
boundary, which comes equipped with an oriented root edge lying on the
boundary. Let us first briefly recall the notion of planar quadrangulations
with a boundary.

\subsection{Planar maps and quadrangulations with a boundary}
\label{sec:planarmaps}
A finite planar map is a finite connected graph properly
embedded in the two-dimensional sphere, that is, in such a way that 
edges intersect only at their endpoints.  As usual, we regard two such maps as
being equivalent, if they differ only by a homeomorphism that preserves the
orientation of the sphere.

The faces of a planar map are the connected components of the complement of
the union of its edges. The degree of a face is the number of its incident
edges, where, as usual, an edge that lies entirely in a face is counted
twice.

A planar map is a {\it quadrangulation with a boundary}, if all
faces have degree four, except possibly one face called the root face,
which can have an arbitrary (even) degree. The edges surrounding the root
face form the boundary of the quadrangulation. We do not require the
boundary to be a simple curve. 

The size of a quadrangulation with a boundary is the number (possibly
infinite) of its non-root or inner faces. The size of the boundary, which
is also called the perimeter of the map, is given by the degree of the root
face. Note that since quadrangulations are bipartite, the perimeter is an
even number.

Provided the perimeter is non-zero, in which case the map is seen as a
single vertex map, we root such a quadrangulation by specifying one
distinguished oriented edge on the boundary, in such a way that the root
face lies to the right of that edge. The origin of the root edge is called
the root vertex. We write $\cQ_f$ for the set of all finite (rooted)
quadrangulations with a boundary. Of course, if the perimeter of an element
$q\in\cQ_f$ is equal to four, we may view $q$ more naturally as a
quadrangulation without boundary. 

Equipped with the usual graph distance $\dgr$, the vertex set $V(\m)$ of a
rooted planar map $\m$ is a pointed metric space. Let us next recall
the so-called {\it local topology} on the set $\cQ_f$ (or more generally,
on the set of finite rooted maps).

Given a rooted planar map $\m$ with root vertex $\varrho$, we denote by
$\cb_r(\m)$ for $r\geq 0$ the combinatorial ball of radius $r$, that is, the
submap of $\m$ containing all vertices $v$ of $\m$ with $\dgr(\varrho,v)\leq
r$, together with the edges of $\m$ connecting such vertices. Now if $\m$ and
$\m'$ are two rooted planar maps, the local distance between $\m$ and $\m'$ is
defined as
$$
\dmap(\m,\m') = \left(1+\sup\{r\geq 0:\cb_r(\m)=\cb_r(\m')\}\right)^{-1}.
$$
The local topology is the topology induced by $\dmap$, and we write $\cQ$
for the completion of $\cQ_f$ with respect to $\dmap$. Elements in
$\cQ\backslash\cQ_f$ are called {\it infinite quadrangulations with a boundary}.

The $\UIHPQ$ $Q_{\infty}^\infty$ is a random (rooted) infinite
quadrangulation with an infinite boundary, which can be obtained as a local
limit of random elements in $\cQ_f$, in the following ways.

Firstly, let $Q_n^{\sigma}$ be uniformly chosen among all rooted
quadrangulations of size $n$ with a boundary of size $2\sigma$,
$\sigma\in\N=\{1,2,\ldots\}$. Curien and Miermont proved in~\cite{CuMi}
that with respect to $\dmap$,
$$
Q_n^{\sigma}\xrightarrow[n \to \infty]{(d)}Q_{\infty}^{\sigma},\quad
Q_{\infty}^{\sigma}\xrightarrow[\sigma \to
\infty]{(d)}Q_{\infty}^{\infty}.$$ Here, $Q_{\infty}^{\sigma}$ is the
so-called uniform infinite planar quadrangulation with a boundary of length
$2\sigma$, see~\cite{CuMi} for a precise description. Similar convergences
hold if $Q_n^{\sigma}$ is chosen uniformly among all rooted
quadrangulations of size $n$ with a simple boundary of size $2\sigma$, that
is, if $Q_n^{\sigma}$ is a uniform rooted quadrangulation of the
$2\sigma$-gon with $n$ inner faces. In this case, the limiting map when first
$n\rightarrow\infty$ and then $\sigma\rightarrow\infty$ is the uniform
infinite planar quadrangulation with a simple boundary $\UIHPQ^{(s)}$, as
alluded to above (see~\cite{AnCu} for details).

Secondly, the $\UIHPQ$ $Q_{\infty}^{\infty}$ arises also as the local limit
of random elements in $\cQ_f$ when the boundary grows simultaneously with
the size of the map. More specifically, assume that $\sigma_n$ grows much
slower than $n$. Then it is shown in~\cite{BaMiRa} that
$$
Q_n^{\sigma_n}\xrightarrow[n \to \infty]{(d)}Q_{\infty}^{\infty}.$$

In~\cite{CuMi}, the $\UIHPQ$ $Q_{\infty}^\infty$ is constructed from an 
extended Schaeffer-type mapping applied to a so-called {\it uniform
  infinite treed bridge} of infinite length, and we will recall and work
with this construction in the following section.

A new construction of the $\UIHPQ$ which is better suited to study the
metric balls around the root has recently been given
in~\cite{CaCu}. Although we will work with the first construction, we
adopt some notation from there.  

In the following section, we introduce certain deterministic objects which
encode (non-random) infinite quadrangulations {\it via} a Schaeffer-type
mapping. Randomized versions of these objects will then encode the
$\UIHPQ$.
\subsection{A Schaeffer-type construction}
\subsubsection{Well-labeled trees and infinite treed bridges}
Recall the definition of a (rooted) finite planar tree $\tau$, see,
e.g.,~\cite{LGMi}. We denote by $|\tau|$ the number of its edges and write
$V(\tau)$ for the vertex set of $\tau$.

A {\it well-labeled tree} $(\tau,\ell)$ is a pair of a rooted planar tree
$\tau$ and integer labels $\ell=(\ell(u))_{u\in V(\tau)}$, which are
attached to the vertices of $\tau$, according to the following rule:
Whenever $u,v\in V(\tau)$ are connected by an edge, then $|\ell(u)-\ell(v)|
\leq 1$.

For $k\in\Z$, we let $\LT_k$ be the set of all finite well-labeled plane trees,
whose root is labeled $k$. The set of all well-labeled plane trees is
denoted $\LT=\cup_{k\in\Z}\LT_k$.

As in~\cite{CuMi} or~\cite{CaCu}, we will work with so-called {\it treed
  bridges.} We will only need their infinite versions, which we define next.
First, an {\it infinite bridge} is a two-sided sequence
$\br=(\br(i):i\in\Z)$ with $\br(0)=0$ and $|\br(i+1)- \br(i)|=1$. An index
$i$ for which $\br(i+1)=\br(i)-1$ is called a {\it down-step} of $\br$. The
set of all down-steps of $\br$ is denoted $\DS(\br)$.
\begin{defn}
  We call {\it infinite treed bridge} a pair $(\br, T)$, where $\br$ is an
  infinite bridge and $T$ is a mapping from $\DS(\br)$ to $\LT$ with the
  property that $T(i)\in \LT_{\br(i)}$, i.e., $T(i)$ is a well-labeled tree
  whose root has label $\br(i)$. 
\end{defn}
We write $\TB^{-\infty}$ for the set of all infinite treed bridges which
have the property that $\inf_{i\in\Z_+} \br(i) = -\infty$ and $\inf_{i\in\Z_-}
\br(i) = -\infty$, where $\Z_+=\{0,1,2,\ldots\}$, $\Z_-=\{\ldots,-2,-1,0\}$.
\subsubsection{The Bouttier-Di Francesco-Guitter mapping}
\label{sec:BDG-mapping}
We now construct a mapping $\Phi$, which we call the Bouttier-Di
Francesco-Guitter mapping, that sends elements in $\TB^{-\infty}$ to infinite
quadrangulations with an infinite boundary. The uniform infinite
half-planar quadrangulation $\UIHPQ$ is then obtained from applying $\Phi$
to a random element $(\br_\infty,T_\infty)$ in $\TB^{-\infty}$, whose law
we specify in the next section.

We stress that usually (e.g., in~\cite{CuMi}, or in~\cite{CaCu}), the
Bouttier-Di Francesco-Guitter mapping is first introduced as a bijection
between finite versions of treed bridges and (rooted and pointed)
finite-size quadrangulations with a boundary. Then it is argued
that the mapping can be extended to elements in $\TB^{-\infty}$, yielding
infinite quadrangulations. However, since we will here only work with
infinite quadrangulations, we directly describe the mapping as a function
$$\Phi:\TB^{-\infty}\longrightarrow \cQ.$$

Let $(\br,T)\in\TB^{-\infty}$. It is convenient to work with the following
representation of $(\br, T)$ in the plane: We identify
$\br=(\br(i):i\in\Z)$ with the labeled bi-infinite line, which is obtained
from connecting the neighboring vertices of $\Z$ by edges and assigning to
$i\in\Z$ the label $\br(i)$. Then we graft a proper embedding of the tree
$T(i)$ for $i\in\DS(\br)$ to the vertex $i$ in the upper half-plane, by
identifying the root of $T(i)$ with the vertex $i$. See
Figure~\ref{fig:CornerBijection}. Note our small abuse of notation: We
denote here by $i\in\DS(\br)$ an index of $\br$ as well as a vertex of the
representation of $\br$.

The vertex set of such a representation of $(\br,T)$ is therefore given by
$\Z$ and the union of the tree vertices of $T(i)$, $i\in\DS(\br)$, where we
interpret the root of $T(i)$ and the vertex $i\in\Z$ as one and the same
vertex. Following the wording of~\cite{CaCu}, we call the vertices which
belong to the trees $T(i)$, $i\in\DS(\br)$, {\it real vertices}, and the
vertices $j\in\Z$ above which no trees are grafted, i.e., the vertices $j$
that do not correspond to down-steps of $\br$, {\it phantom vertices}. A
{\it corner} of (the representation of) $(\br, T)$ is an angular sector
between two consecutive edges, in the clockwise contour or left-to-right
order. Henceforth we shall consider only {\it real corners}, i.e., corners
that are incident to real vertices and lie in the upper half-plane. By a
small abuse of notation, given a vertex $v\in T(i)$, $i\in\DS(\br)$, we
shall simply write $\ell(v)$ for its label, and we let $\ell(c)=\ell(v)$ if
$c$ is a corner incident to $v$.

\begin{figure}[ht]
	\begin{center}
		\includegraphics[scale=0.9]{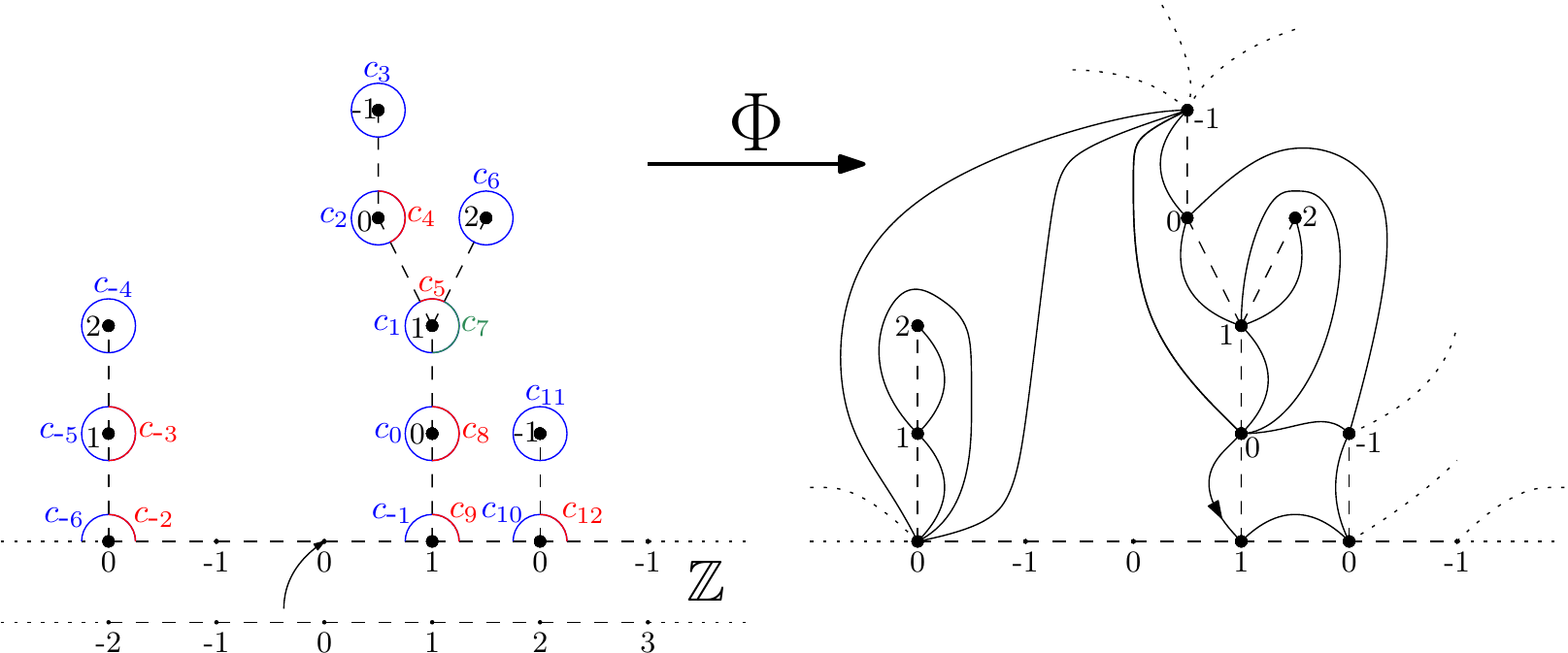}
	\end{center}
	\caption{The Bouttier-Di Francesco-Guitter mapping. Vertex $0$ of
          $\Z$ is indicated by an arrow. The visible trees are
          attached to the vertices $-2$, $1$ and $2$ of $\Z$, which are
          labeled $0$, $1$ and $0$, respectively. These vertices correspond to
          down-steps of the bridge.}
	\label{fig:CornerBijection}
\end{figure}

We now consider the bi-infinite sequence of corners $(c_i)_{i\in\Z}$
obtained from ordering the real corners of $(\br, T)$ according to the
left-to-right order, where we agree that $c_0$ is the left-most real corner
with label $0$, which appears in $T(i)$, $i\in\DS(\br)\cap\Z_+$. See again
Figure~\ref{fig:CornerBijection}. For $i\in\Z$, we denote by $\suc(c_i)$
the first corner among $c_{i+1},c_{i+2},\ldots$, which has label
$\ell(c_i)-1$. Note that such a corner always exists, since
$\inf_{i\in\Z_+}\br(i)=-\infty$. We call $\suc(c_i)$ the {\it successor} of
$i$. As indicated on the right side of Figure~\ref{fig:CornerBijection}, we
draw for every $i\in\Z$ an arc between the corner $c_i$ and $\suc(c_i)$ in
the upper half-plane, in such a way that arcs do only possibly intersect at
their endpoints. We finally erase the phantom vertices and the edges that
stem from the representation of $(\br, T)$. We obtain a locally finite
quadrangulation $M$ with an infinite boundary $\partial M$, which we root
in the (oriented) edge that corresponds to the first step of the bridge to
the right of $0$. A detailed explanation of this correspondence is given in
the next section. In other words, the root face that lies to the right of
the root edge has infinite degree, and the edges surrounding it form the
(infinite) boundary $\partial M$ of the map.

We let $\Phi((\br, T))=M$ be the rooted infinite quadrangulation with an
infinite boundary obtained in this way.

\subsubsection{Identification of the boundary}
\label{sec:BoundaryMap}
If we identify $\Z$ with the bi-infinite line by connecting neighboring
vertices with an edge, then the Bouttier-Di Francesco-Guitter mapping
establishes a one-to-one correspondence between the edges of $\Z$ and those
of the boundary $\partial M$ of $M=\Phi((\br, T))$, as it is visible in
Figures~\ref{fig:CornerBijection} and~\ref{fig:MaxGeodBoundary}. More
precisely, for a given $(\br,T)$, we define a function
$$
\varphi :\Z\rightarrow V(\partial M)
$$
as follows: Vertex $i\in\Z$ of the representation of $\br$ (which is labeled
$\br(i)$) is mapped to itself, if $i$ is a real vertex. By definition, this
is the case if and only if $i\in\DS(\br)$. Otherwise, we search for the
next real corner to the right of $i$ which has label $\br(i)$, and define
$\varphi(i)$ to be the vertex incident to it. Then the edge $\{i,i+1\}$ of
$\Z$ corresponds to a unique edge from $\varphi(i)$ to $\varphi(i+1)$ of
$\partial M$, and the assignment is one-to-one. Instead of being more
formal, we refer to Figure~\ref{fig:MaxGeodBoundary}.

We will call $\varphi(\Z_-)$ and $\varphi(\Z_+)$ the {\it left} and {\it
  right part } of the boundary of $M$, respectively. Of course, $\partial M
= \varphi(\Z_-)\cup\varphi(\Z_+)$. Moreover, $M$ is rooted in the
(oriented) edge between $\varphi(0)$ and $\varphi(1)$.

\subsubsection{Construction of the $\normalfont{\UIHPQ}$}
\label{sec:constr-UIHPQ}
Recall the definition of $\LT_k$ for $k\in\Z$. Let $\rho_k$ be the
Boltzmann measure on $\LT_k$ given by $\rho_k((\tau,\ell))=
12^{-|\tau|}/2$. The measure $\rho_k$ is the law of a so-called {\it
  uniformly labeled critical geometric Galton-Watson tree}. This means that
if $(\tau,\ell)$ is distributed according to $\rho_k$, then $\tau$ has the
law of a Galton-Watson tree with a geometric offspring distribution of
parameter $1/2$. Moreover, conditionally on $\tau$,
$\ell:V(\tau)\rightarrow\Z$ is the random labeling of $\tau$ such that the
root receives label $k$, and independently for each edge $e=\{u,v\}$ of $\tau$,
$\ell(u)-\ell(v)$ is uniformly distributed over $\{-1,0,1\}$. We refer,
e.g., to~\cite[Section 2.2]{LGMi} for more details.

Let $\br_\infty=(\br_{\infty}(i):i\in\Z)$ be a two-sided simple symmetric random walk
with $\br_\infty(0)=0$, that is, $(\br_{\infty}(i):i\in\Z_+)$ and
$(\br_{\infty}(i):i\in\Z_-)$ are two independent simple symmetric random walks
starting from $0$. 

Conditionally on $\br_\infty$, define a (random) function
$T_\infty:\DS(\br_\infty)\rightarrow\LT$ by letting $T_\infty(i)$ for
$i\in\DS(\br_\infty)$ be a well-labeled tree with
law $\rho_{\br_\infty(i)}$, independently in $i\in\DS(\br_\infty)$.

We call the random element $(\br_\infty,T_\infty)$ of $\TB^{-\infty}$ a
{\it uniform infinite treed bridge}.
 \begin{defn}
   The $\UIHPQ$ $Q_\infty^\infty=(V(Q_\infty^\infty),\dgr,\varrho)$ is the
   random infinite quadrangulation with an infinite boundary obtained from
   applying the Bouttier-Di Francesco-Guitter mapping to a uniform infinite
   treed bridge $(\br_\infty,T_\infty)$, i.e.,
$$Q_\infty^\infty=\Phi\left((\br_\infty,T_\infty)\right).$$
\end{defn}
We will write $\ell_\infty(v)$ for the label of a vertex $v\in
V(T_\infty(i))$, $i\in\DS(\br_\infty)$, which we also identify with a vertex
of $Q_\infty^\infty$ {\it via} the Bouttier-Di Francesco-Guitter mapping.

\subsection{Geodesics in the $\normalfont{\UIHPQ}$}
Let $G=(V(G),E(G))$ be a graph. A {\it geodesic} in $G$ is a path of
possibly infinite length, which visits a sequence (or chain) of vertices
$\gamma= (\gamma(0),\gamma(1),\ldots)$ of $G$ such that for $i,j\in\Z_0$ for
which $\gamma$ is defined, $\dgr(\gamma(i),\gamma(j))=|i-j|$. An infinite
geodesic $\gamma$ with $\gamma(0)=v\in V(G)$ is called a {\it geodesic ray}
started at $v$.

Note that we view a geodesic as a sequence of concatenated edges. In
particular, if $G$ is a non-simple graph as in the case of the $\UIHPQ$, a
geodesic is usually not specified by its vertices alone.

Let $(\br, T)\in\TB^{-\infty}$ be an infinite treed bridge. We will now
define particular geodesic rays in the infinite quadrangulation $\Phi((\br,
T))$. Recall the definition of the sequence of corners $(c_i)_{i\in\Z}$
obtained from ordering the real corners of $(\br, T)$ according to the
contour order, as well as the definition of the successor-mapping; see
Section~\ref{sec:BDG-mapping}. We write $\suc^{(i)}$ for the $i$-fold
composition of the successor-mapping and denote by $\mathcal{V}(c)$ the
vertex incident to the 
corner $c$.

\begin{defn}[Maximal geodesic]
\label{def:maxgeod}
Let $(\br, T)\in\TB^{-\infty}$, and let $v\in V(\Phi((\br, T)))$ be a
vertex of the quadrangulation associated to $(\br,T)$. Let $c$ be the
leftmost (real) corner of $(\br,T)$ incident to $v$. Then the {\it maximal
  geodesic} started at $v$ is given by the chain of vertices incident to
the iterated successors of $c$, that is, $\gmax^v(0)=v$, and then for
$i\in\N$,
$$
\gmax^v(i)=\mathcal{V}(\suc^{(i)}(c)),
$$
and with edges connecting $\suc^{(i)}(c)$ to $\suc^{(i+1)}(c)$ for $i\in\Z_+$.
\end{defn}

We will simply write $\gmax$ for the maximal geodesic started from the root
$\varrho$. See Figure~\ref{fig:MaxGeodBoundary} for an illustration of the
maximal geodesic in the $\UIHPQ$. It is a direct consequence of the
definition that maximal geodesics finally coalesce. Indeed, consider the
first vertex incident to a corner $c_i$ for $i\in\Z_+$, which is visited by
$\gmax^v$. Let $v'$ be the first vertex incident to a corner $c_j$, $j\geq
i$, which is visited by $\gmax$. Then $v'$ is also visited by $\gmax^v$,
and from that moment on, $\gmax^v$ and $\gmax$ coincide.

Of special interest is the class of {\it proper geodesics}, which
generalizes the construction of maximal geodesics, in the sense that the
connecting edges do not necessarily emanate from leftmost corners.

\begin{defn}[Proper geodesic]
  A geodesic ray $\gamma$ is \textit{proper}, if for every $i\in
  \Z_+$, $$\ell(\gamma(i+1))=\ell(\gamma(i))-1.$$
\end{defn} 
It turns out that in the $\UIHPQ$, almost surely every geodesic ray is
proper. This fact has already been proved in~\cite{CaCu}, but we will give
an alternative proof in Corollary~\ref{cor:propergeod}. In particular, it
makes sense to call maximal geodesics {\it leftmost geodesics}. In
Section~\ref{sec:sparseness}, we shall also consider {\it minimal} or {\it
  rightmost geodesics}.
\begin{figure}[ht]
	\begin{center}
		\includegraphics[scale=0.8]{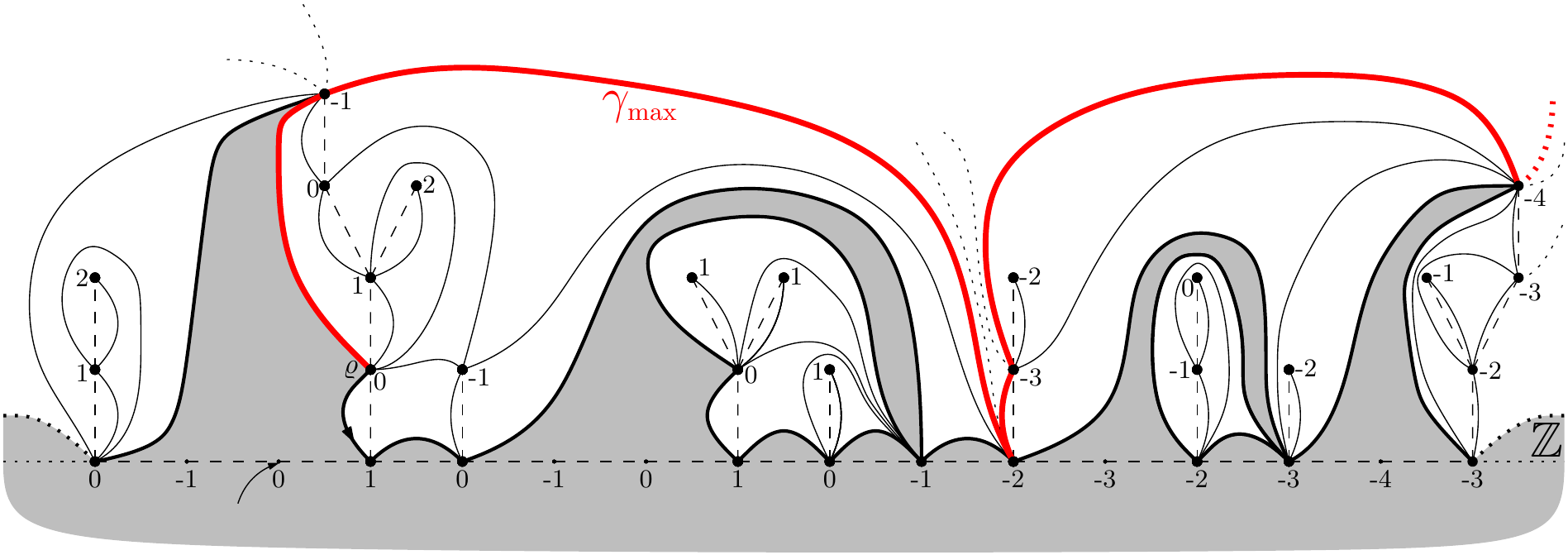}
	\end{center}
	\caption{The $\UIHPQ$ and its maximal geodesic $\gmax$.}
	\label{fig:MaxGeodBoundary}
\end{figure}

\section{Proof of the main results}
\label{sec:mainproofs}
To begin with, let us describe our {\bf general strategy} for proving
Theorem~\ref{thm:geod4}. 

We will first show that the maximal geodesic $\gmax$ hits both parts of the
boundary of the $\UIHPQ$ infinitely many times, see
Proposition~\ref{prop:intersection} below. For that purpose, we will study
the sets $\cR_+$ and $\cR_-$ of intersection times of $\gmax$ with the
right and left part of the boundary. It turns out that both $\cR_+$ and
$\cR_-$ are regenerative sets. Moreover, we find a representation of these
sets in terms of the infinite treed bridge encoding the $\UIHPQ$, which involves the
minimal label attained in the trees between two subsequent minima of the
bridge. The crucial step is formulated as Proposition~\ref{prop:Delta_0}
below, where we compute the exact distribution of such a minimal
label. Once we know that $\gmax$ touches both parts of the boundary
infinitely often, we also know that every geodesic ray must cross $\gmax$
infinitely many times. From this, we readily deduce that any geodesic ray
is proper, as it was already shown in~\cite[Proposition 4.8]{CaCu} for
geodesic rays started from the root vertex, by means different from
ours. Since any proper geodesic ray lies finally in between $\gmax$ and the
boundary, an appeal to Proposition~\ref{prop:intersection} allows us to
conclude the proof of Theorem~\ref{thm:geod4}.

We first introduce some more notation. Let $(\br,T)\in\TB^{-\infty}$ be an infinite
treed bridge. For
$j\in\Z_+$, we write
$$
H_j(\br)=\inf\{m\in \Z_+:\br(m)=-j\},\quad H'_j=\sup\{m\in
  \Z_- : \br(m)=-j\}
$$
for the first time $\br$ hits $-j$ to the right of zero or to the left of zero,
respectively. Note that both
$H_j(\br)$ and $H'_j(\br)$ are finite for each $j\in\Z_+$, almost surely.

Moreover, for $i\in\DS(\br)$, we write $\ell_i=\left(\ell_i(u)\right)_{u\in
  V(T(i))}$ for the labels of the vertices of the tree $T(i)\in
\LT_{\br(i)}$. Recall that if $r$ is the root vertex of $T(i)$, then
$\ell_i(r)=\br(i)$.

For $j\in\Z_+$, we let
\begin{align*}
\Delta_j((\br,T)) &= \max_{i\in\DS(\br)\cap[H_j,H_{j+1})}
-\left(\min_{u\in V(T(i))}\ell_i(u)+j\right),\quad\textup{and}\\
\Delta'_j((\br,T))&=\max_{i\in\DS(\br)\cap[H'_{j+1},H'_j)}
-\left(\min_{u\in V(T(i))}\ell_i(u)+j\right),
\end{align*}
where $H_j=H_j(\br_\infty)$, and $H'_j=H'_j(\br_\infty)$. In words,
$\Delta_j((\br,T))\in\Z_+$ is the absolute value of the minimal label
shifted by $|\br(H_j)|=j$ in the trees $T(i)$ that are attached to the
infinite bridge $\br$ on $[H_j,H_{j+1})$. A similar interpretation holds
for $\Delta'_j((\br,T))$. We simply write $\Delta_j$ and $\Delta'_j$ for
the random numbers $\Delta_j((\br_\infty,T_\infty))$ and
$\Delta'_j((\br_\infty,T_\infty))$, where $(\br_\infty,T_\infty)$ is a
uniform infinite treed bridge as specified in
Section~\ref{sec:constr-UIHPQ}. The strong Markov property shows that
$\Delta_j$ has the same law as $\Delta_0$, and $\Delta'_j$ has the same law as
$\Delta'_0$, for each $j\in\Z_+$.  As we show next, their distributions can
be computed explicitly.
\begin{prop}
\label{prop:Delta_0}
We have for $m\in\N$,
$$
\P(\Delta_0\geq m)=\frac{1}{m+1},\quad\hbox{and}\quad
\P(\Delta'_0\geq m)=\frac{1}{m+3}.
$$
\end{prop}

\begin{proof} We first consider $\Delta_0$. The statement for $\Delta'_0$
  will then follow from a symmetry argument. Let $m\in\N$. We set
  $g(m) = \P(\Delta_0< m)$. Moreover, let
  $h(m)=\P(\min_{u\in V(\tau)}\ell(u)>-m)$, where $(\tau,\ell)$ is
  distributed according to $\rho_0$; see Section~\ref{sec:constr-UIHPQ}. We
  decompose the path of $\br$ on $[0,H_1)$ into its excursions above $1$,
  as shown in Figure~\ref{fig:HintProof}.  For $\Delta_0$ to be smaller
  than $m$, the labels in every excursion above $1$ have to be larger than
  $-(m+1)$, while the minimal label of the tree grafted to the last step of
  the excursion has to be larger than $-m$.

A standard application of the strong Markov property shows that these
excursions, shifted by $-1$, have the same law as $\br$ on $[0,H_1)$,
so that the quantity $g(m)$ satisfies the recursive equation
\begin{equation}\label{eq:archdecomp}
	g(m)=\frac{1}{2}h(m)\sum_{k=0}^\infty\left(\frac{1}{2}g(m+1)\right)^k=
\frac{h(m)}{2-g(m+1)}.
\end{equation}
We stress that~\eqref{eq:archdecomp} is in spirit of the arch
decomposition as described in Section V$.4.1$ of~\cite{FlSe}; see also
$(2.1)$ and $(2.2)$ of~\cite{BoGu} for related decompositions.

\begin{figure}[ht]
	\begin{center}
		\includegraphics[scale=0.9]{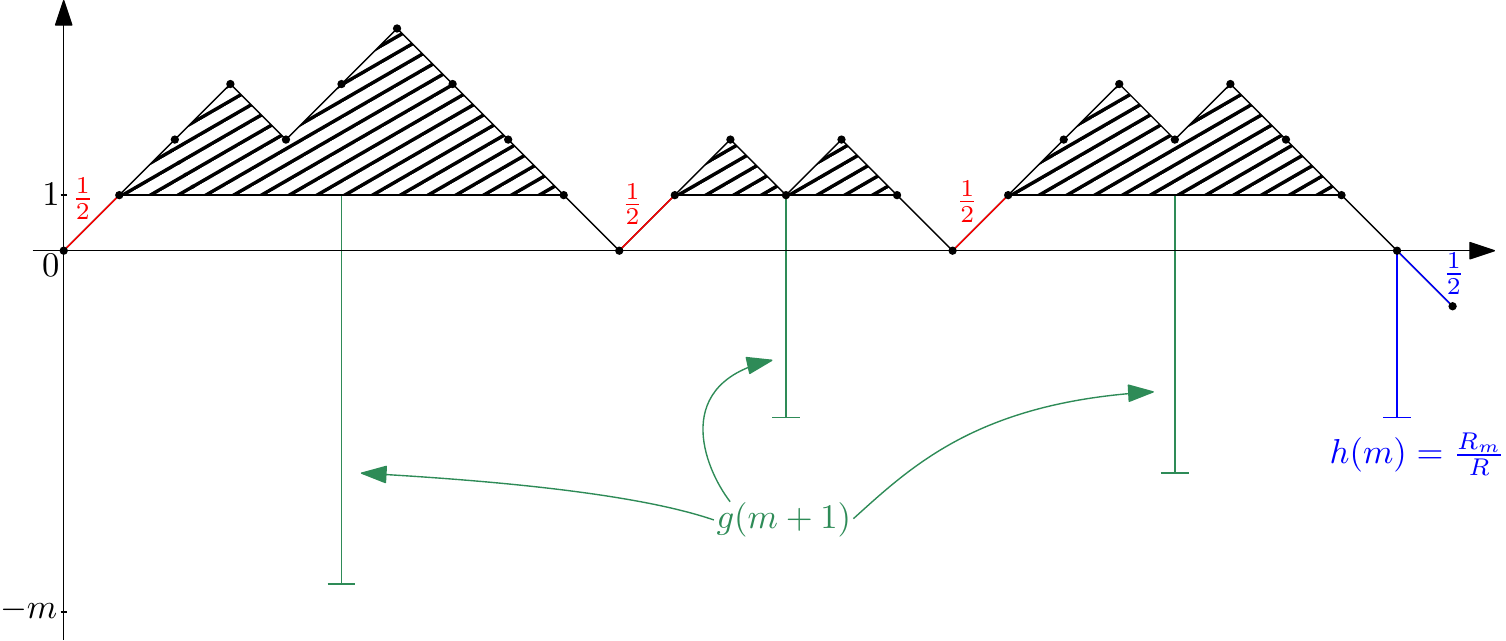}
	\end{center}
	\caption{The decomposition of the probability $g(m)$.}
	\label{fig:HintProof}
\end{figure}

From the Bouttier-Di Francesco-Guitter bijection for quadrangulations of a
finite size, see, e.g.,~\cite{BoDFGu}, well-labeled trees are in bijection
with rooted and pointed quadrangulations, the pointed vertex being at
distance $\min_{u\in V(\tau)}\ell(u)-1$ from the root. In~\cite{BoGu}, the
generating function for quadrangulations with weight $g_4$ per face and
distance less than or equal to $m$ between the root and the pointed vertex,
called the {\it distance-dependent two-point function} and denoted $R_m$,
is proved to satisfy (see~\cite[(6.18)]{BoGu})
\[R_m=R\frac{(1-y^m)(1-y^{m+3})}{(1-y^{m+1})(1-y^{m+2})},\] where
$R=R(g_4)=\lim_{m\rightarrow\infty}R_m$ is the generating function of
rooted and pointed quadrangulations with weight $g_4$ per face, and
$y=y(g_4)$ is the solution of the so-called characteristic equation
(see~\cite[(6.17)]{BoGu}).  In our special case corresponding to a critical
weight per face given by $g_{4,\textrm{cr}}=1/12$, the solution of the characteristic
equation simplifies to $y=1$. Taking the limit $y\uparrow 1$ in the last
display, this implies
\[R_m=R\frac{m(m+3)}{(m+1)(m+2)}.\] Since the partition function is given
by $R$, we therefore get
\begin{equation}
\label{eq:hm}
h(m)= \P\left(-\min_{u\in V(\tau)}\ell(u)\leq
m-1\right)=\frac{R_m}{R}=1-\frac{2}{(m+1)(m+2)}.
\end{equation}
By the way, we note that $h(m)$ has already been calculated before
in~\cite[Proposition 2.4]{ChDu}; see Remark~\ref{rem:Delta_0}
below. Letting $f(m)=\P(\Delta_0\geq m)=1-g(m)$, we obtain from
\eqref{eq:archdecomp} and the last display
$$
f(m)-f(m+1)+f(m)f(m+1)=\frac{2}{(m+1)(m+2)}\quad\textup{for all }m\in\N.
$$
Our claim about $\Delta_0$ now follows from the following
\begin{lemma}
\label{lem:nlsystem}
Consider the non-linear system
\begin{equation}
\label{eq:nlequation}
\left\{
\begin{array}{rl}
f(m)-f(m+1)+f(m)f(m+1)&=\frac{2}{(m+1)(m+2)}\,\quad\textup{for all } m\in\N,\\
f(0) &= 1,\\
\lim_{m\rightarrow\infty}f(m) &= 0.
\end{array}\right.
\end{equation}
Then the only solution $f$ of~\eqref{eq:nlequation} with $f(m)\in (0,1)$ for
all $m\in\N$ is given by $f(m)= 1/(m+1)$, $m\in\Z_+.$
\end{lemma}
\begin{proof}
It is elementary to check that $f(m)=1/(m+1)$, $m\in\Z_+$, is a solution
of~\eqref{eq:nlequation} with $f(\N)\subset (0,1)$, so it remains to show uniqueness.
We first prove the following statement:
\begin{equation}
\label{eq:nlsystem-1}
   \parbox{0.9\textwidth}{If $f_1,f_2:\Z_+\rightarrow (0,1)$ 
are two solutions of~\eqref{eq:nlequation} such that $f_1(m)<f_2(m)$ for some
$m\in\N$, then $f_1(m+k)<f_2(m+k)$ for all $k\in\Z_+$.}
  \end{equation}
Indeed, assume $f_1(m)<f_2(m)$ for some $m\in\N$. We show that then also
$f_1(m+1)<f_2(m+1)$. Since $f_1$ is a solution of~\eqref{eq:nlequation}, we can
use~\eqref{eq:nlequation} to express $f_1(m+1)$ in terms of $f_1(m)$ and obtain
$$
f_1(m+1)=\frac{(m+1)(m+2)f_1(m)-2}{(m+1)(m+2)(1-f_1(m))}<\frac{(m+1)(m+2)f_2(m)-2}{(m+1)(m+2)(1-f_2(m))}=f_2(m+1).
$$
An iteration of the argument shows $f_1(m+k)<f_2(m+k)$ for all $k\in\Z_+$ and
hence~\eqref{eq:nlsystem-1}. 

Now assume there are two solutions $f_1,f_2:\Z_+\rightarrow
(0,1)$ of~\eqref{eq:nlequation} with $f_1\neq f_2$. Then there exists
$\varepsilon>0$ and $m\in\N$ such that $f_2(m)-f_1(m)>\varepsilon$ or
$f_1(m)-f_2(m)>\varepsilon$. By symmetry, we may assume the former. Since both $f_1$ and $f_2$
solve~\eqref{eq:nlequation}, we obtain for their difference
\begin{equation}
\label{eq:nlequation-2}
f_2(m)-f_1(m) - \left(f_2(m+1)-f_1(m+1)\right)+f_2(m)f_2(m+1)-f_1(m)f_1(m+1) =0.
\end{equation}
By assumption, $f_2(m)-f_1(m)>\varepsilon$, which implies
by~\eqref{eq:nlsystem-1} that $$f_2(m)f_2(m+1)-f_1(m)f_1(m+1)>0.$$
Therefore, we obtain from~\eqref{eq:nlequation-2} that also
$f_2(m+1)-f_1(m+1)>\varepsilon$.
Iterating the argument, we see $\lim_{m\rightarrow\infty}f_2(m)\geq
\varepsilon$, a contradiction to $\lim_{m\rightarrow\infty}f_2(m) =0$.
\end{proof} 
\begin{figure}[ht]
	\begin{center}
          \includegraphics[scale=0.8]{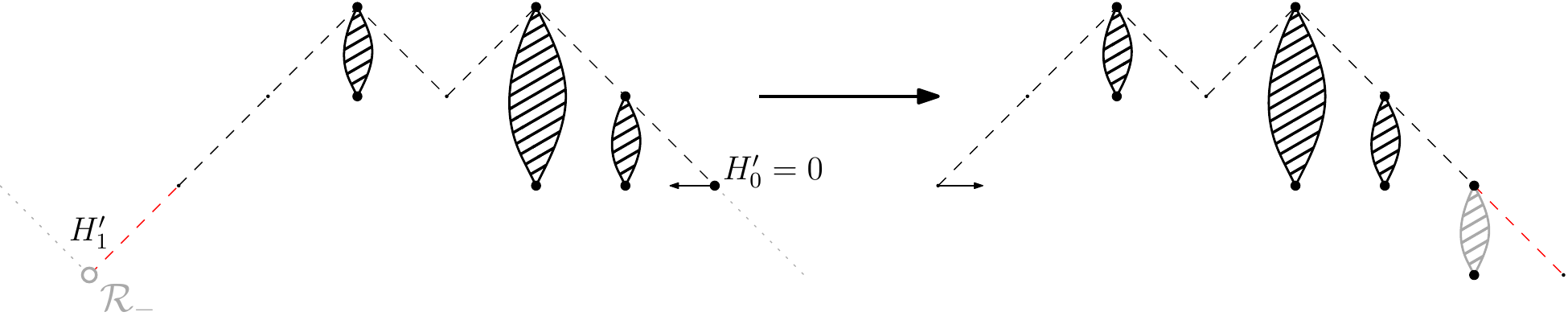}
              \end{center}
              \caption{The symmetry argument between excursions of
                $\br_\infty$.}
	\label{fig:SymmetryArgument}
\end{figure}
We continue the proof of Proposition~\ref{prop:Delta_0} and turn to the
distribution of $\Delta'_0$. By time-reversal, $(\br_\infty(i):H'_1<i\leq
0)$ has the same law as $(\br_\infty(i):0\leq i < H_1)$. Moreover,
down-steps $i$ of $(\br_\infty(i):H'_1<i\leq 0)$ belong to
$\DS(\br_\infty)$, and as shown in Figure~\ref{fig:SymmetryArgument},
independent trees with law $\rho_{\br_\infty(i)}$ are assigned to
them. However, $H'_1-1$ is an up-step of the bridge, where no tree is
attached to, while $H_1-1$ is a down-step. As a consequence, if we modify
$T_\infty$ by attaching an independent tree with law $\rho_0$ to $H'_1-1$,
the whole process $(\br_\infty,T_\infty)$ has the same law on $[0,H_1]$ as
on $[H'_1,0]$. Thus, for every $m \in \Z_+$,
$$\P(\Delta_0<m)=\P(\Delta'_0<m)h(m),$$ which gives from the first part of
the proposition that for every $m\in\N$,
$$\P(\Delta'_0\geq m)=\frac{1}{m+3}.$$ This concludes the proof of
Proposition~\ref{prop:Delta_0}.
\end{proof}

\begin{remark}
\label{rem:Delta_0}
Note that as an intermediate step in the proof of
Proposition~\ref{prop:Delta_0}, we explicitly compute the distribution of a
minimal label in a well-labeled tree $(\tau,\ell)$ with law $\rho_0$, cf.
Display~\eqref{eq:hm}. As it was pointed out to us by the referee, the
calculation of $h(m)$ was already performed in~\cite[Proposition
2.4]{ChDu}. In~\cite[Lemma 12]{CuMeMi}, it is (only) shown that the tail
distribution behaves asymptotically like $2/m^2$ as $m$ tends to
infinity. The methods of~\cite{CuMeMi} rely on the fact that the label
function $\ell$ has its continuous analog in the so-called Brownian
snake. We stress that for our purpose, the asymptotic tail behavior of the
minimal label of $(\tau,\ell)$ would not provide enough information, see
Remark~\ref{rem:intersection-spine} below.
\end{remark} 

We let $Q_\infty^\infty=\Phi((\br_\infty,T_\infty))$ be the $\UIHPQ$
defined in terms of a uniform infinite treed bridge
$(\br_\infty,T_\infty)$. Recall the identification of $\Z$ with $\partial
Q_\infty^\infty$ {\it via} the function $\varphi$. Our presentation is now
similar to that of~\cite[Section 3.2.2]{CuMeMi}. From now on, $\gmax$ will
denote the maximal geodesic in the $\UIHPQ$ emanating from the root
$\varrho$.  By construction of the Bouttier-Di Francesco-Guitter mapping
and by definition of $\gmax$, a vertex $\varphi(j)\in\partial
Q_\infty^\infty$ for $j\in\Z_+$ is hit by $\gmax$ if and only if it is
incident to the first (real) corner in contour order starting from $c_0$
with label $\ell_\infty(\varphi(j))$, i.e., if and only if
$$
\min\{\ell_{\infty,i}(v): v\in V(T_\infty(i)),\, i\in\DS(\br_\infty),\,
0\leq i\leq j-1\} > \br_\infty(j),
$$
where $\ell_{\infty,i}$ denotes the labeling of $T_\infty(i)$.
In particular, if we introduce the set of intersection times of the
maximal geodesic with the right boundary of the $\UIHPQ$,
$$
\cR_+=\{j\in\Z_+:\gmax(j)\in\varphi(\Z_+)\},
$$
we have
$$
\cR_+=\Z_+\backslash\bigcup_{j\geq 0}\left(j,j+\Delta_j\right.].
$$
See Figure~\ref{fig:GeodMaxBis} for an illustration.  It follows from the
last display that $\cR_+$ can be represented as the set
$\{G_0+G_1+\dots+G_n: n\in\Z_+\}$, where $G_0=0$, and $(G_i:i\in\N)$ is a
sequence of i.i.d. variables with
\begin{equation}
\label{eq:defG}
G_{1}=\inf\left\lbrace i>0:\max\{j+\Delta_j:0\leq j\leq i-1\}<i\right\rbrace.
\end{equation}
In particular, $\cR_+$ is a discrete regenerative set, and the renewal
theorem shows that the asymptotic frequency of $\cR_+$ is given by
\begin{equation}
\label{eq:R-asymp-freq}
|\cR_+|=\lim_{n\rightarrow\infty}\frac{\#\cR_+\cap\{1,\ldots,n\}}{n}=\frac{1}{\E\left[G_1\right]}.
\end{equation}

\begin{figure}[ht]
	\begin{center}
		\includegraphics[scale=0.8]{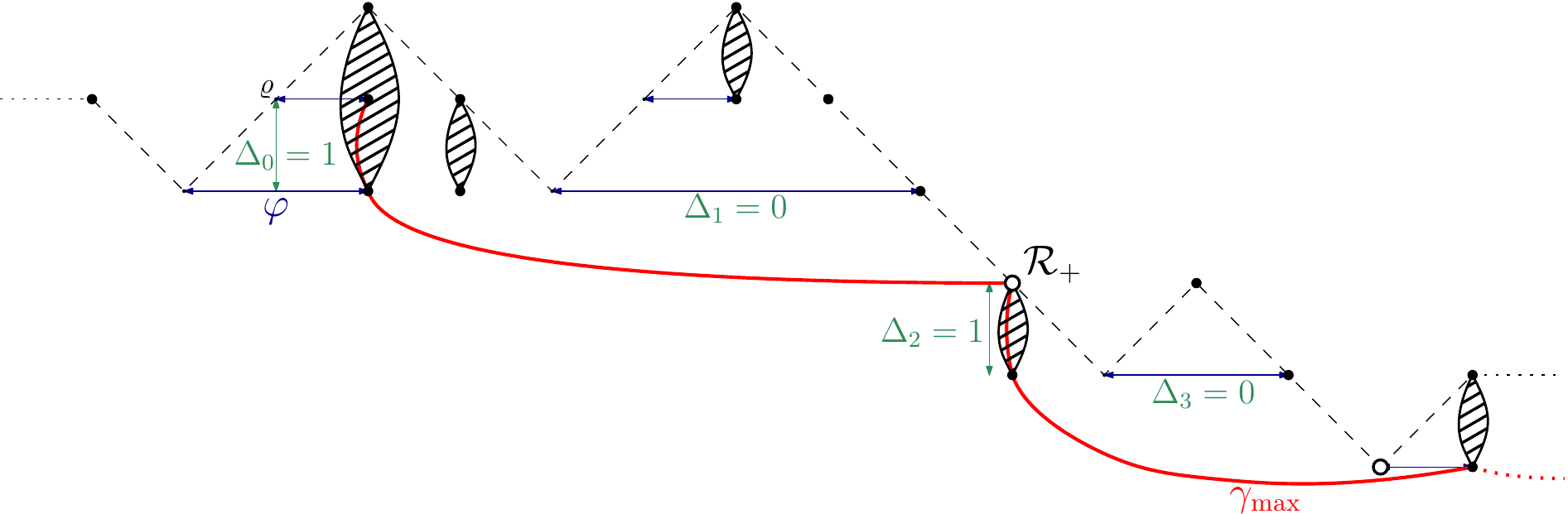}
	\end{center}
	\caption{Alternative representation of the $\UIHPQ$ and its maximal
          geodesic $\gmax$ as depicted in
          Figure~\ref{fig:MaxGeodBoundary}. (Trees are represented by the
          striped almonds, whose lower endpoints indicate the minimal label
          in the corresponding tree.)}
	\label{fig:GeodMaxBis}
\end{figure}

We will also study the set of intersection times of the maximal geodesic
with the left part of the boundary, 
$$
\cR_-=\{j\in\Z_+:\gmax(j)\in\varphi(\Z_-)\}.
$$ Using again the construction of the $\UIHPQ$ {\it via} the Bouttier-Di
Francesco-Guitter mapping, we can express this set as
$$
\cR_-=\Z_+\backslash\bigcup_{j\geq 0}\left(j,j+\Delta'_j\right.].
$$ 
Similarly to $\cR_+$, we have  
$\cR_-=\{G'_0+G'_1+\dots+G'_n:n\in\Z_+\}$, where again $G'_0=0$, and
$(G'_i:i\in\N)$ is an i.i.d. family of random variables specified by
\begin{equation}
\label{eq:defG'}
G'_{1}=\inf\left\{i>0:\max\{j+\Delta'_j:0\leq j\leq i-1\}<i\right\}.
\end{equation}
Note that $(G'_i:i\in\N)$ is also independent of $(G_i:i\in\N)$. Indices
$j\in\cR_-$ correspond to (certain) up-steps of the bridge and thus to
phantom vertices.  Then, the associated vertex $\varphi(j)$ is incident to
the first (real) corner in contour order starting from $c_0$ with label
$\ell_\infty(\varphi(j))$ and is therefore visited by the maximal geodesic.

We now formulate the key proposition of this paper.
\begin{prop}\label{prop:intersection}
We have for $i\in\N$,
$$
\P(i\in\cR_+)=\frac{1}{i+1},\quad\hbox{and}\quad\P(i\in\cR_-)=\frac{3}{i+3}.
$$  
Also, almost surely, both $\cR_+$ and $\cR_-$ are infinite sets,
and the maximal geodesic $\gmax$ hits the left as well as the right part of
the boundary of the $\UIHPQ$ infinitely many times.  However, this happens
with asymptotic frequency zero: $|\cR_+|=0$ and $|\cR_-|=0$ almost surely.
  \end{prop}
\begin{proof}
  The arguments for $\cR_+$ and $\cR_-$ are entirely similar. Let us
  first consider $\cR_+$. By Proposition~\ref{prop:Delta_0} in the last
  equation, we have for $i\in\N$
\begin{align*}
  \P(i\in\cR_+)&=\P\left(\max\{j+\Delta_j:0\leq j\leq i-1\}<i\right)\\
  &=\prod_{j=0}^{i-1}\left(1-\P\left(\Delta_0\geq
      i-j\right)\right)=\prod_{j=1}^{i}\left(1-\P\left(\Delta_0\geq
      j\right)\right)\\
  &=\exp\left(\sum_{j=1}^i\ln\left(1-\frac{1}{j+1}\right)\right)=\frac{1}{i+1}.
\end{align*}
We deduce from the last display that
$$
\E\left[\#\cR_+\right]=\sum_{i=0}^\infty\P(i\in\cR_+) = \infty.
$$
From this, we readily infer that $\#\cR_+=\infty$ almost surely: Indeed, if
the contrary were true, then necessarily $G_1=\infty$ with some probability
$\alpha>0$. However, then the number of points in $\cR_+$ different from
$0$ is geometrically distributed with parameter $\alpha$, a contradiction
to $\E[\#\cR_+]=\infty.$ The fact that $|\cR_+|=0$ follows
from~\eqref{eq:R-asymp-freq} and Proposition~\ref{prop:Delta_0}. Concerning
$\cR_-$, we simply have to replace $\Delta_0$ by $\Delta'_0$ in the above
argumentation. An application of Proposition~\ref{prop:Delta_0} shows
$\P(i\in\cR_-)=3/(i+3)$, and the remaining statements for $\cR_-$ follow
from the same reasoning as above.
\end{proof}
Albeit being
infinite, the sets $\cR_+$ and $\cR_-$ are rather sparse. We will make 
this more precise in Section~\ref{sec:sparseness}.
\begin{remark}
\label{rem:intersection-spine}
The last proposition should be compared with Proposition 15
of~\cite{CuMeMi}.  Proposition~\ref{prop:Delta_0} has its counterpart in
Lemma 14 of~\cite{CuMeMi}, where it is shown that the quantity
corresponding to $\P\left(\Delta_0\geq m\right)$ behaves asymptotically
like $2/m$ for $m$ tending to infinity. The multiplicative factor being
larger than $1$, this implies in the context considered there that the
number of intersections between the maximal geodesic and the spine of the
$\UIPQ$ is {\it finite} almost surely. Here, in the setting of the
$\UIHPQ$, we find an exact formula for $\P\left(\Delta_0\geq m\right)$,
which came somewhat as a surprise and is the key observation that leads to
Proposition~\ref{prop:intersection}. We emphasize that an equivalent of the
form $\P\left(\Delta_0\geq m\right)\sim 1/m$ would not be sufficient to
deduce that $\cR_+$ is an infinite set, and the same for $\cR_-$.

For the intersection of the independent regenerative
sets $\cR_+$ and $\cR_-$, we have for $i\in\Z_+$
$$
\P\left(i\in\cR_+\cap\cR_-\right)=\frac{3}{(i+1)(i+3)},
$$
and with arguments similar to those in the proof of
Proposition~\ref{prop:intersection}, we get that the left and right
boundary of the $\UIHPQ$ intersect {\it finitely} many times. Actually, we
have here obtained a new proof of the fact shown in~\cite{CuMi} that the
$\UIHPQ$ contains a well-defined {\it core}, that is an infinite submap
homeomorphic to the half-plane. In~\cite{CuMi}, the well-definedness of the
core was obtained by a limiting argument, starting from an infinite
quadrangulation with a simple boundary of a finite (randomized) size, while
we prove this result directly in terms of the $\UIHPQ$.
\end{remark}

Note that since any maximal geodesic finally coincides with $\gmax$,
Proposition~\ref{prop:intersection} implies that any maximal geodesic has
infinitely many intersection points with the left and right part of the
boundary of the $\UIHPQ$.  We now prove that all geodesic rays in
 the
$\UIHPQ$ are proper. Theorem~\ref{thm:geod4} will then readily follow. The
following result was already established in Proposition 4.8 of~\cite{CaCu}
for geodesic rays started from the root vertex, by similar but different
arguments.

\begin{corollary}[see Proposition 4.8 of~\cite{CaCu}]
\label{cor:propergeod}
  Almost surely, all geodesics rays in the $\UIHPQ$
  $Q_\infty^\infty=\Phi((\br_\infty,T_\infty))$ are proper.
\end{corollary}

\begin{proof} Here, we propose a simple proof that uses the result of
  Proposition~\ref{prop:intersection}. Let $\eta$ be an infinite
  self-avoiding path
  in $Q_\infty^{\infty}$. Since by the above
  proposition, the maximal geodesic $\gmax$ intersects the left and right
  boundary infinitely often, the path $\eta$ also intersects $\gmax$
  infinitely often, as indicated by Figure \ref{fig:ProperGeod}. 

  Let $\gamma$ be a geodesic ray in $Q_\infty^\infty$. To simplify
  notation, we assume that $\gamma$ starts at the root $\varrho$
  (if not, one should consider the maximal geodesic started from
  $\gamma(0)$). The above remark applied to $\eta=\gamma$ shows that
  $\gamma$ and $\gmax$ intersect infinitely many times.  Let $(u_i:
  i\in\Z_+)$ be the sequence of vertices at which $\gamma$ and $\gmax$
  intersect, with $u_0=\varrho$ and such that $u_i$ is visited before
  $u_j$ if $i<j$. Then, for every $i\in\Z_+$, by definition of the maximal
  geodesic, $$\dgr(u_{i+1},u_i)=\ell_\infty(u_i)-\ell_\infty(u_{i+1}).$$
  Because labels differ at most by one between neighboring vertices of
  the map, the length of the segment of $\gamma$ between $u_i$ and
  $u_{i+1}$ is at least
  $\ell_\infty(u_i)-\ell_\infty(u_{i+1})=\dgr(u_{i+1},u_i)$. Therefore,
  equality must hold since $\gamma$ is a geodesic, and this implies
  that labels always decrease by one as $\gamma$ goes from $u_i$ to $u_{i+1}$,
  meaning that $\gamma$ is proper on this segment. This finishes the
  proof.
\end{proof}

\begin{figure}[ht]
	\begin{center}
	\includegraphics[scale=0.9]{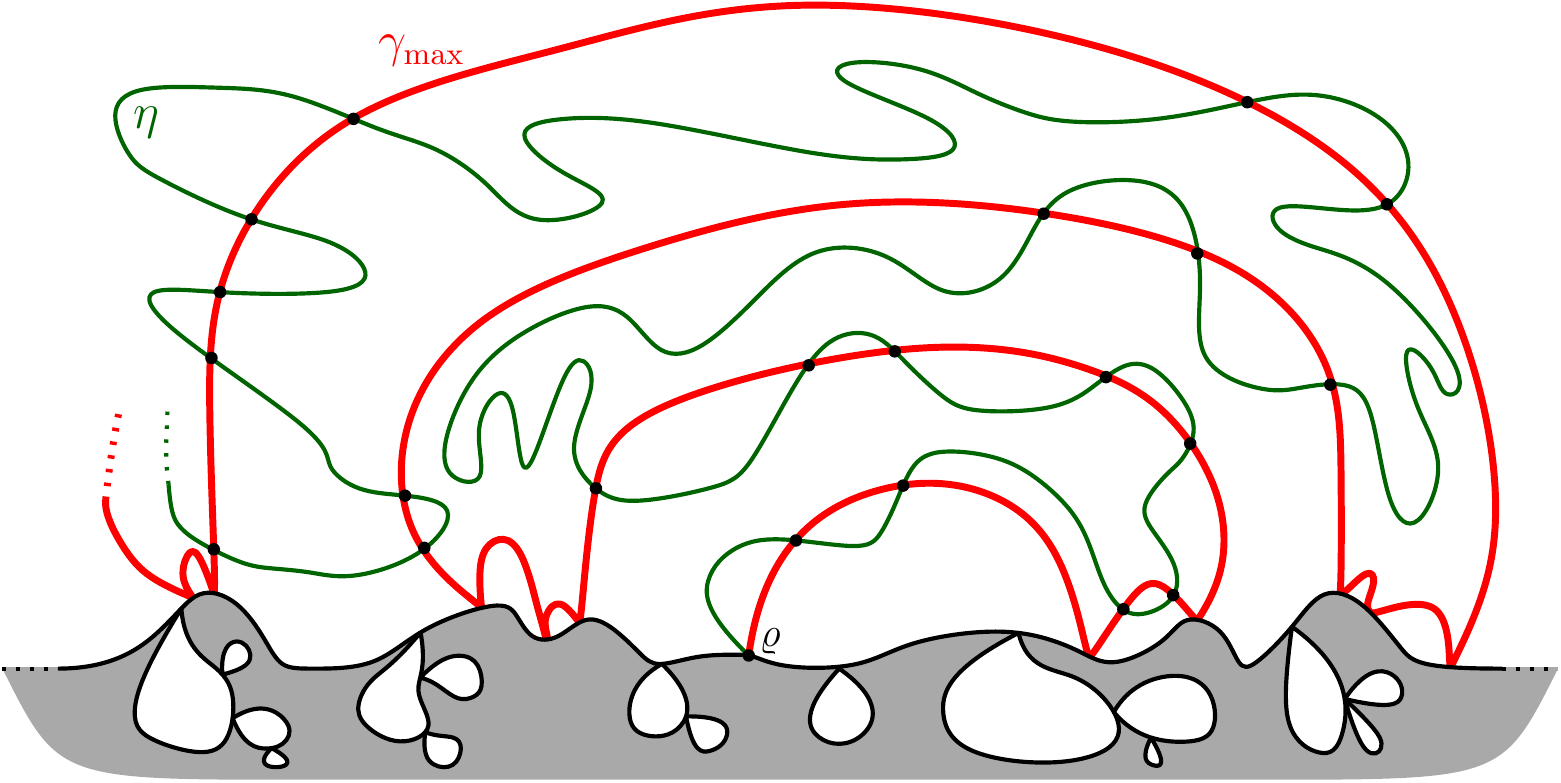}
	\end{center}
	\caption{The infinite path $\eta$ intersects $\gmax$ infinitely
          many times. }
	\label{fig:ProperGeod}
\end{figure}

The proof of Theorem~\ref{thm:geod4} is now an immediate consequence of our
foregoing considerations.
\begin{proof}[Proof of Theorem~\ref{thm:geod4}]
  For the purpose of the proof, we will assume that the $\UIHPQ$ is given
  in terms of a uniform infinite treed bridge,
  $Q_\infty^\infty=\Phi((\br_\infty,T_\infty))$. Let $\gamma$ be a geodesic
  ray. By Corollary~\ref{cor:propergeod}, we can assume that $\gamma$ is
  proper. Hence each edge of $\gamma$ connects a real corner of
  $(\br_\infty,T_\infty)$ to its successor. Now let $n_0\in\Z_+$ be the
  first instant when the maximal geodesic emanating from $v=\gamma(0)$ hits
  the left part of the boundary. We have seen above that $n_0$ is finite
  almost surely. By definition, $\gmax^v$ always connects leftmost corners
  to their successors. In particular, the embedding of $\gamma$ in the
  upper half-plane (in terms of the Bouttier-Di Francesco-Guitter mapping)
  lies in between $(\gmax^v(n):n\geq n_0)$ and the boundary of the map, see
  Figure~\ref{fig:GeodFromV}. Otherwise said, vertices of the right part of
  the boundary which are visited by $(\gmax^v(n):n\geq n_0)$ are also visited by
  any other proper geodesic started at $v$. Since $\gmax^v$ coincides after
  a finite number of steps with $\gmax$, the maximal geodesic started from
  the root $\varrho$, Proposition~\ref{prop:intersection} concludes the
  proof. \end{proof}

\begin{figure}[ht]
	\begin{center}
	\includegraphics[scale=0.9]{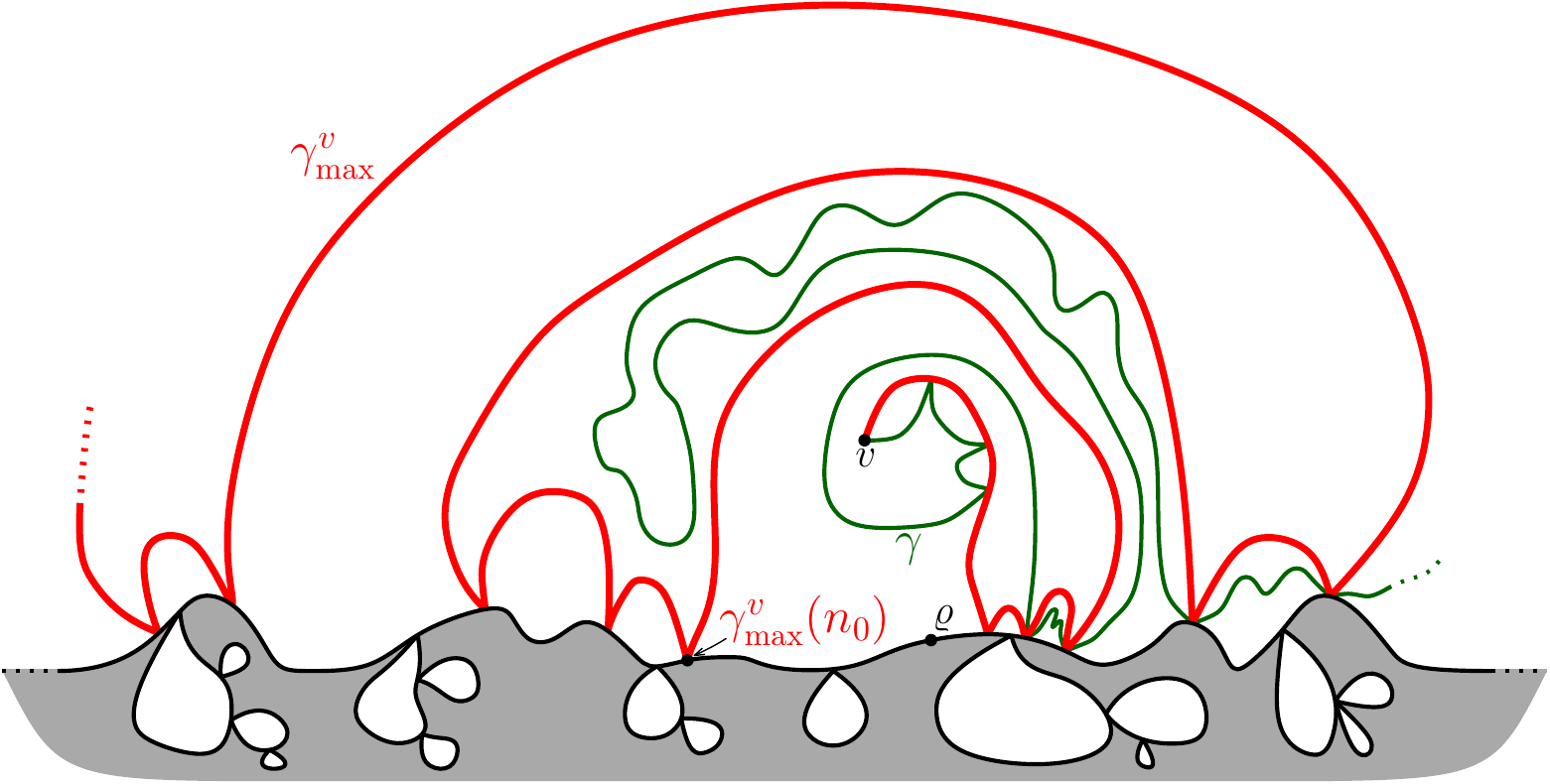}
	\end{center}
	\caption{The geodesic $\gamma$ lies in between $(\gmax^v(n):n\geq n_0)$ and the boundary of the map.}
	\label{fig:GeodFromV}
\end{figure}

\begin{corollary}
\label{cor:UIHPQ-simple}
Theorem~\ref{thm:geod4} remains true if the $\UIHPQ$ is replaced by its
analog with a simple boundary, the $\UIHPQ^{(s)}$.
\end{corollary}
\begin{proof}
  We give only a sketch proof, since the statement is essentially a
  consequence of the pruning construction of the $\UIHPQ^{(s)}$ out of the
  $\UIHPQ$, as explained in~\cite{CuMi} (see, in particular, Proposition 6
  in this work). Roughly speaking, after removing the finite quadrangulations
  which hang off from the pinch-points of the boundary of the $\UIHPQ$, a
  core consisting of a unique infinite quadrangulation with an infinite
  simple boundary remains, which has, after a rooting operation, the law of
  the $\UIHPQ^{(s)}$.  Since geodesics started from the core of the
  $\UIHPQ$ do not visit the finite quadrangulations that are attached to
  the pinch-points of the boundary (the pinch-points 
would be visited twice),
  Theorem~\ref{thm:geod4} applies to the $\UIHPQ^{(s)}$ as well.
\end{proof}

\section{Sparseness of the intersections with the boundary}
\label{sec:sparseness}
From Theorem~\ref{thm:geod4}, we know that every geodesic ray in the
$\UIHPQ$ hits the boundary infinitely many times. The goal
of this section is to show that these hitting times and hitting points are, however,
sparsely distributed, in a way that we will make precise in
Proposition~\ref{prop:geod4sparse} below.

For that purpose, recall that the sets $\cR_+$ and $\cR_-$ of intersection
times of the maximal geodesic with the right and left part of the boundary,
respectively, admit the representation
$$
\cR_+=\{G_0+G_1+\dots+G_n:n\in\Z_+\},\quad\cR_-=\{G'_0+G'_1+\dots+G'_n:n\in\Z_+\},  
$$
where $G_0=G'_0=0$, and the families $(G_i:i\in\N)$ and $(G'_i:i\in\N)$
consist of i.i.d. random variables specified by~\eqref{eq:defG}
and~\eqref{eq:defG'}, respectively. We find the following asymptotic behavior.
\begin{lemma} 
\label{lem:Gasymp}
For $m$ tending to infinity, we have
$$
\P\left(G_1= m\right)\sim\frac{1}{m\ln^2 m},\quad \P\left(G'_1=
  m\right)\sim\frac{1}{3m\ln^2 m}.
$$
\end{lemma}
\begin{proof}
  We first look at $G_1$. For $n\in\Z_+$, let $u_n=\P(n\in\cR_+)$,
  $f_n=\P(G_1=n)$. Note that $f_0=0$ and $u_0=1$. A classical decomposition
  (see, e.g., Section XIII.3 in~\cite{Fe}) of $u_n$ according to the
  smallest non-zero element in $\cR_+$, i.e., according to the value of
  $G_1$, gives the recursive relation
$$
u_n=f_1u_{n-1}+f_2u_{n-2}+\dots+f_nu_0,\quad n\in\N.
$$
For the generating functions $U(s)=\sum_{n\geq 0}u_ns^n$ and
$F(s)=\sum_{n\geq 0}f_ns^n$, the last relation implies
$$
U(s)= \frac{1}{1-F(s)},\quad |s|<1.
$$
Using that $\P(n\in\cR_+)=1/(n+1)$, see
Proposition~\ref{prop:intersection}, we obtain for $0<|s|<1$ the expression
$U(s)=-(1/s)\ln\left(1-s\right)$. Therefore,
$$F(s)= 1-s\ln^{-1}\left(\frac{1}{1-s}\right),\quad |s|<1.$$ Standard
singularity analysis, see, e.g., $(24)$ on page $387$ of~\cite{FlSe},
yields the first claim. For $G'_1$, we use that $\P(n\in\cR_-)=3/(n+3)$,
see again Proposition~\ref{prop:intersection}. For the generating function
$H(s)=\sum_{n\geq 0}\P(G'_1=n)s^n$, this gives similarly to above the
relation
$$
H(s)=1-(s^3/3)\left(\ln\left(\frac{1}{1-s}\right)-s^2/2-s\right)^{-1},\quad
|s|<1.
$$
Since $1-H(s)\sim (1/3)(1-F(s))$ as $s\rightarrow 1$, an application
of~\cite[Theorem IV.4]{FlSe} finishes the proof of the second claim.
\end{proof}
\begin{remark}
  The above lemma should be compared with the asymptotics of the returns to
  zero of a recurrent two-dimensional random walk $S=(S_n : n\in \Z_+)$. For
  concreteness, let us assume that $S$ is the simple symmetric random walk on
  $\Z^2$ started from zero. Let $\cR$ be the
  regenerative set of return times to zero of $S$. One has the
  representation $\cR=\{G_0+G_1+\cdots+G_n:n\in \Z_+\}$,
  where $G_0=0$, and $(G_i : i\in\N)$ are the waiting times between two consecutive
  returns. Then, as $m\rightarrow\infty$, we get the asymptotics
  (\cite[Chapter III, Section 16, Example 1]{Sp})
  $$\P(m \in \cR)\sim \frac{1}{\pi m} \quad \text{and} \quad \P(G_1=m)\sim \frac{\pi}{m \ln^2 m}.$$
\end{remark}
Coming back to geodesics in the $\UIHPQ$, we note that Lemma~\ref{lem:Gasymp}
gives precise quantitative information on the number of steps between two
consecutive visits of the boundary by the maximal geodesic $\gmax$. The
distance measured along the boundary between two consecutive times of
intersection is bounded from below by the number of steps of $\gmax$ in
between these times.

In the proof of Theorem~\ref{thm:geod4}, we have seen that any geodesic ray
$\gamma$ is finally enclosed between $\gmax$ and the boundary of the
$\UIHPQ$. {\it A priori,} this does not exclude the existence of a
geodesic ray that visits the boundary with a much higher frequency than
$\gmax$. We will now argue that this is not the case.

In this regard, it is convenient to introduce the {\it minimal geodesic} in
the $\UIHPQ$ emanating from the root $\varrho$. Given
$(\br_\infty,T_\infty)$ and $v$ a real vertex of $(\br_\infty,T_\infty)$,
we write $c^{(r)}(v)$ for the rightmost corner incident to $v$. Note that
in the list of corners $(c_i)_{i\in\Z}$ as specified in
Section~\ref{sec:BDG-mapping}, $c^{(r)}(v)$ appears as the last corner
incident to $v$ (in the lexicographical order).

The minimal geodesic $\gmin$ starting from $\varrho$ is then given by the
chain of vertices $\gmin(0)=\varrho$, and for $i\in\N$,
$$
\gmin(i)=\mathcal{V}\left(\suc\left(c^{(r)}(\gmin(i-1))\right)\right).
$$
The edge set of $\gmin$ is given by the edges connecting
$c^{(r)}(\gmin(i))$ to $c^{(r)}(\gmin(i+1))$ for $i\in\Z_+$.

Similarly to above, one defines for $\gmin$ the (random) sets of
intersection times with the right and left part of the boundary, respectively,
$$\cR^{\textup{\tiny min}}_+=\{j\in\Z_+:\gmin(j)\in\varphi(\Z_+)\},\quad
\cR^{\textup{\tiny min}}_-=\{j\in\Z_+:\gmin(j)\in\varphi(\Z_-)\}.$$

The following symmetry argument shows that the random set
$\cR_+^{\textup{\tiny min}}$ (defined in terms of $\gmin$) has the same law
as $\cR_-$ (defined in terms of $\gmax$). Consider the mapping that
associates to a (possibly infinite) rooted planar map $\m$ its ``mirror"
$\overleftarrow{\m}$, which is obtained from applying a symmetry with respect to any
line of the plane, and reversing the orientation of the root edge. This
transformation is better understood by seeing a planar map as a gluing
of polygons: Then, the map $\overleftarrow{\m}$ is obtained by reversing the
orientation of the polygons forming $\m$, and that of the root edge.  Now,
it is seen that this transformation preserves the uniform measure on
quadrangulations with a fixed size and perimeter, and thus the law of the
$\UIHPQ$. Finally, recall that the maximal and minimal geodesics started at
the root vertex are also the leftmost and rightmost geodesics,
respectively, and are thus exchanged by the ``mirror" mapping. It follows
that $\cR_+^{\textup{\tiny min}}$ and $\cR_-$ have the same law, and, by
the same symmetry argument, $\cR_-^{\textup{\tiny min}}$ has the same law as $\cR_+$.

As a direct consequence of the way edges are drawn in the Bouttier-Di
Francesco-Guitter construction of the $\UIHPQ$, and of the fact that every
geodesic ray is proper, see Corollary~\ref{cor:propergeod}, we notice that
any geodesic ray $\gamma$ lies finally in between $\gmax$ and
$\gmin$. Indeed, this is the case from the first vertex on hit by $\gamma$
that is incident to a corner $c_i$ with $i\in\Z_+$. See
Figure~\ref{fig:MinGeodBoundary} for an illustration. From the
constructions of $\gmax$ and $\gmin$, we see that $\cR_+$ is a subset of
$\cR_+^{\textup{\tiny min}}$, and similarly $\cR_-^{\textup{\tiny min}}$ is
a subset of $\cR_-$.

\begin{figure}[ht]
	\begin{center}
	\includegraphics[scale=0.8]{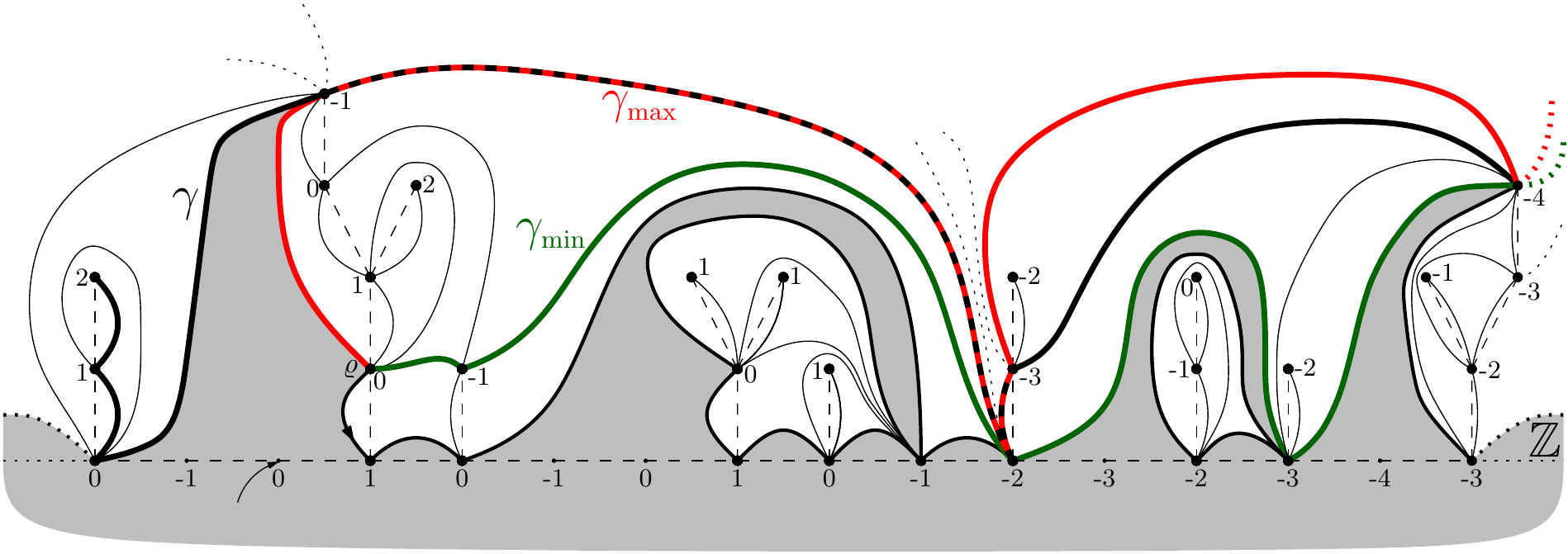}
	\end{center}
	\caption{The geodesic $\gamma$ (black bold) started at the leftmost
          vertex labeled $2$ is enclosed by $\gmin$
          (green bold) and $\gmax$ (red bold) after it
          first hits the latter (at the topmost vertex labeled $-1$).}
	\label{fig:MinGeodBoundary}
\end{figure}

We collect our observations in the following proposition, which should be
read as an extension to Theorem~\ref{thm:geod4}. For simplicity, we
restrict ourselves to geodesic rays emanating from the root vertex; see,
however, the remark below the proposition.
\begin{prop}
\label{prop:geod4sparse}
Almost surely, for any geodesic ray $\gamma=(\gamma(i):i\in\Z_+)$ in the
$\UIHPQ$ $Q_\infty^\infty=\Phi((\br_\infty,T_\infty))$
started from the root vertex, we have the inclusions
$$
\cR_+\cup\cR_-^{\textup{\tiny min}} \subseteq
\{i\in\Z_+:\gamma(i)\in\partial Q_\infty^\infty\}\subseteq \cR_+^{\textup{\tiny min}}\cup\cR_-.
$$
The random sets $\cR_+$ and $\cR_-^{\textup{\tiny min}}$ (as well as
$\cR_+^{\textup{\tiny min}}$ and $\cR_-$) have the same
law. The distance $\delta$ between two consecutive times in $\cR_+$ (or
$\cR_-^{\textup{\tiny min}}$) exhibits the tail behavior $\P(\delta>m)\sim
1/\ln m$ as $m\rightarrow\infty$, whereas the distance $\delta'$ between
two consecutive times in $\cR_+^{\textup{\tiny min}}$ (or $\cR_-$)
satisfies $\P(\delta'>m)\sim 1/(3\ln m)$.
\end{prop}
\begin{remark}
Let $\gamma=(\gamma(i):i\in\Z_+)$ be any geodesic ray in the $\UIHPQ$ (not
necessarily started from the root vertex), and let $v$ be the first vertex to the right of the root $\varrho$ which is
hit by both $\gamma$ and $\gmax$. Let $n$, $n'\in\Z_+$ such that
 $\gamma(n)=\gmax(n')=v$, and set $j=n-n'$. Now consider the shifted geodesic 
$\gamma_j(i)=\gamma(i+j)$, $i\geq \max\{0,-j\}$. On the event of
full probability where $\gamma$, $\gmax$ and $\gmin$ are proper, we 
have the inclusions
$$
\left(\cR_+\cup\cR_-^{\textup{\tiny min}}\right)\backslash
\{0,\ldots,n'\}\subseteq \{i\geq \max\{0,-j\}:\gamma_j(i)\in\partial
Q_\infty^\infty\}  \subseteq {\cR}_+^{\textup{\tiny min}}\cup\cR_-.
$$
\end{remark}

\section{Extension to the uniform infinite half-planar triangulation and
  further remarks}
\label{sec:UIHPT}
The uniform infinite half-planar triangulation $\UIHPT$ is an infinite
triangulation of the half-plane. A variation with a simple boundary
(i.e., the triangular analog to the $\UIHPQ^{(s)}$) was introduced by Angel
in~\cite{An}.

In this part, we will argue that the intersection times with the
boundary of geodesics in the $\UIHPT$ behave in way comparable to that in the
$\UIHPQ$. More precisely, it turns out that the right part of the boundary
is hit by the maximal geodesic started from the root with exactly the same
frequency as in the $\UIHPQ$, whereas the distribution of the hitting times
of the left part of the boundary undergoes a slight change.

In order to avoid too much repetition, we will not treat the case of the
$\UIHPT$ in full detail. We will rather argue that the strategy
developed for the $\UIHPQ$ applies to the $\UIHPT$ as well, and then sketch
how the computations have to be modified. Our discussion will therefore
lack a certain rigor, but should enable the reader to fill in the remaining
details. In order to make a clear distinction to the $\UIHPQ$, some of our
quantities considered in this section will be decorated with the tilde
sign.

Triangulations, or more generally (rooted and pointed) planar maps with
prescribed face valences, can be encoded in terms of labeled trees called
{\it mobiles}, see~\cite{BoDFGu}. Let us briefly recall the encoding:
First, label each vertex of the map by its distance from the pointed vertex
minus the distance from the pointed vertex to the origin of the root
edge. Put a new vertex without label in the center of each face. Now walk
around each face $F$ in the clockwise order, and look at each of its
incident edges. If for an edge $e$, the label decreases by $1$ when walking
clockwise around $F$, then connect the endpoint of $e$ with the larger
label to the (unlabeled) vertex in the middle of $F$. If the labels of the
endpoints of $e$ are both equal to $n$, say, add a {\it flagged vertex}
with flag $n$ in the middle of $e$ and connect the flagged vertex with two
new edges to the two central vertices of the faces incident to $e$. In the
third case, that is, for edges where the labels increase when walking around
the face $F$, do nothing. See Figure~\ref{fig:mobile}. By removing all the
original edges of the map together with the pointed vertex, one obtains a
{\it mobile}, i.e., a plane tree with three types of vertices: labeled
and unlabeled vertices, and flagged vertices.

Note that by construction, flagged vertices have degree $2$, and unlabeled
vertices are in one-to-one correspondence with the faces of the
map. Moreover, the degree of the corresponding face equals twice the number
of labeled vertices plus the number of flagged vertices that are connected
to the unlabeled vertex in the mobile. In particular, an unlabeled vertex
associated to a triangular face has either three flagged vertices or a
flagged vertex and a labeled vertex incident to it.

\begin{figure}[ht]
 \center
 \includegraphics[scale=2]{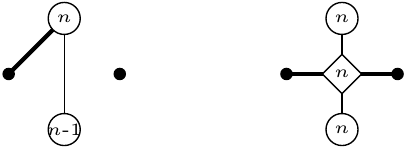}
 \caption{The construction of a mobile. The black dots represent unlabeled vertices of
   the mobile. They are put in the centers of the faces of the map. On the
   left, the bold line represents a mobile edge associated to an edge of
   the map, which connects a vertex labeled $n$ to a vertex
   labeled $n-1$. On the right, the bold line represents a mobile
   edge associated to an edge connecting two vertices with label $n$. The
   flagged vertex is represented by
   a lozenge and receives label $n$, too.}
\label{fig:mobile}
\end{figure}
\begin{figure}[ht]
 \center
 \includegraphics[scale=2]{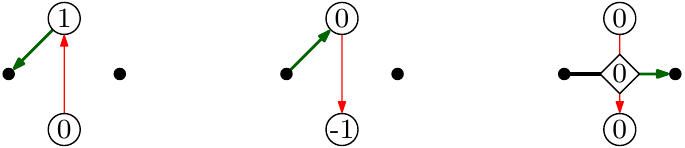}
 \caption{The rooting convention. The red arrow represents the root edge of
   the map, and the green bold arrow is the associated root edge of the mobile.}
\label{fig:mobile-root}
\end{figure}

The root edge of a planar map allows to distinguish a root edge in the
mobile, as depicted in Figure~\ref{fig:mobile-root}. If the root edge of
the map connects two vertices with label $0$, see the right most case 
in Figure~\ref{fig:mobile-root}, it is convenient to regard the encoding
mobile as a pair of {\it half-mobiles} with root flag $0$ each, i.e.,
mobiles which have one distinguished flagged vertex of degree $1$ called the root
flag, which receives label $0$. There is a bijection between rooted pointed
planar maps on the one hand and rooted mobiles and pairs of half-mobiles on
the other hand. We refer to~\cite{BoDFGu} and~\cite{BoGu} for more details.

In terms of generating functions, prescribing the number of faces of a
certain degree $k$ amounts to attach a weight to each face of degree $k$.
For our purpose, we now specialize in triangulations corresponding to the critical weight
sequence $g_k=\gcrit\delta_3(k)$, where $\gcrit=2^{-1}3^{-3/4}$, see,
e.g.,~\cite{Mi}. In this regard, let $R_m$ (or $S_m$) denote the
corresponding generating function of rooted mobiles (or half-mobiles) with
root label (or root flag) $0$, which have their labels all strictly larger
than $-m$ and their flags all larger or equal to $-m$,
cf.~\cite{BoGu}. Letting $R=\lim_{m\rightarrow\infty}R_m$ and
$S=\lim_{m\rightarrow\infty}S_m$, an analysis of $(6.2)$ in~\cite{BoGu}
shows that $R=\sqrt{3}$ and $S=3^{1/4}\left(\sqrt{3}-1\right)$, but this
will be of no importance here.  Note that $R$ and $S$ are the partition
functions for rooted mobiles with root label $0$ and half-mobiles with root
flag $0$, respectively, subject to $g_k=\gcrit\delta_3(k)$.

In order to motivate our construction of the $\UIHPT$, let us first
consider rooted pointed triangulations with a boundary of perimeter 
$n\in\Z_+$. This means that all faces except the root face are triangles,
the root face being incident to $n$ edges (loops and multiple edges are
allowed).  We choose such a triangulation $\m$ according to the Boltzmann
law $\rho(\m)=\gcrit^{\#F(\m)}/Z$, where $F(\m)$ denotes the set of faces
of $\m$ without the root face (which receives no weight), and $Z$ is the
normalizing partition function. Denote by $d$ the distance between the
pointed vertex of $\m$ and the origin of the root edge. Following Section
2.4 of~\cite{BoDFGu}, we associate to the map a (random) path
$(X^{[n]}(i):0\leq i\leq n)$ that encodes the clockwise sequence of
distances minus $d$ between the pointed vertex of the map and the vertices
incident to the root face, with $X^{[n]}(0)$ given by the origin of the
root edge (so that $X^{[n]}(0)=0$).

We decompose the associated mobile around the unlabeled vertex $v$ lying in
the center of the root face of the map. Then each down- or level-step of
$X^{[n]}$ corresponds to a labeled or a flagged vertex, respectively, which
is connected to $v$ by an edge, see Figure~\ref{fig:mobile}. By removing $v$
and its incident edges, one obtains a sequence of rooted mobiles and
half-mobiles. More precisely, a down-step $i$ of $X^{[n]}$ corresponds to a
rooted mobile with root label $X^{[n]}(i)$, while a level-step $i$ of
$X^{[n]}$, that is, an $i$ with $X^{[n]}(i+1)=X^{[n]}(i)$, corresponds to a
half-mobile with root flag $X^{[n]}(i)$. This decomposition is
bijective. Letting $n$ grow, this incites us to define the following
two-sided random walk. Let $C=2\sqrt{R}+S$, and consider
$\tbr_\infty=(\tbr_\infty(i):i\in\Z)$ with $\tbr_\infty(0)=0$, such that
the increments $(\tbr_\infty(i+1)-\tbr_\infty(i):i\in\Z_+)$ are i.i.d.
with law
$$
\P\left(\tbr_\infty(i+1)-\tbr_\infty(i)=\pm 1\right)=\frac{\sqrt{R}}{C},\quad
\P\left(\tbr_\infty(i+1)-\tbr_\infty(i)=0\right)=\frac{S}{C},
$$
and $(\tbr_\infty(i):i\in\Z_-)$ is an i.i.d. copy of
$(\tbr_\infty(i):i\in\Z_+)$. One can show that for fixed
$\ell\in\N$, there is the convergence

$$
\left( X^{[n]}([i]):-\ell\leq i\leq
  \ell\right)\xrightarrow[n\rightarrow\infty]{(d)} \left(\tbr_\infty(i):-\ell\leq
  i\leq \ell\right),
$$
with $[i]$ denoting the representative of $i$ modulo $n$ in $\{0,\ldots,n-1\}$.

We proceed now similarly to the construction of the $\UIHPQ$: Conditionally
on $\tbr_\infty$, we identify $\tbr_\infty$ with $\Z$ equipped with the
labels $(\tbr_\infty(i):i\in\Z)$, and graft independently to each down-step
$i\in \DS(\tbr_\infty)$ a mobile $\theta$ in the upper half-plane with root
label $\tbr_\infty(i)$, distributed according to the Boltzmann measure
$\rho^{(R)}(\theta)=\gcrit^{\#\bullet(\theta)}/R$ (where $\bullet(\theta)$
denotes the set of unlabeled vertices of $\theta$). Moreover, writing
$\LS(\tbr_\infty)$ for the set of level-steps of $\tbr_\infty$, we graft to
each $i\in\LS(\tbr_\infty)$ independently a half-mobile $\theta'$ with root
flag $\tbr_\infty(i)$, distributed according to
$\rho^{(S)}(\theta')=\gcrit^{\#\bullet(\theta')}/S$. We obtain what we call
a {\it uniform infinite mobile bridge} $(\tbr_\infty,\tT_\infty)$, where
$\tT_\infty$ is now a collection of independent mobiles and half-mobiles
associated to the down- and level-steps of $\tbr_\infty$, respectively.

Each realization of $(\tbr_\infty,\tT_\infty)$ is naturally embedded in the
upper-half plane, similarly to the description in
Section~\ref{sec:BDG-mapping}. Recall that mobiles and half-mobiles come
with three types of vertices. We call here a labeled vertex of a mobile or
a half-mobile a {\it real vertex} , and a real corner (of the embedding) is
a corner in the upper half-plane incident to a real vertex. Note that
flagged vertices are not real vertices.

We write $(c_i)_{i\in\Z}$ for the sequence of real corners in the
left-to-right order, again with $c_0$ being the leftmost corner incident to
the root vertex. As in the construction of the $\UIHPQ$,
we now connect each real corner $c_i$ to its successor, that is the first
corner among $c_{i+1},c_{i+2},\ldots$ with label
$\ell(c_i)-1$. Additionally, we connect both corners of the flagged
vertices to the corresponding next real corner in the contour order with the {\it same}
label. See Figure~\ref{fig:BijTriang} for an illustration.

We finally erase the unlabeled vertices and the flagged vertices,
interpreting the two outgoing arcs from a flagged vertex which we added as
a single edge. We also erase all the edges and non-real vertices that stem
from the representation of $(\tbr_\infty,\tT_\infty)$ in the plane. We
obtain what we call the {\it uniform infinite half-planar triangulation}
$\UIHPT$. The bi-infinite line $\Z$ can again be identified with the
boundary of the $\UIHPT$. In particular, it makes sense to speak of the
left or right part of the boundary. We root the $\UIHPT$ according to the
convention described in Section~\ref{sec:BoundaryMap}.

\begin{remark}
  We stress that the above construction does not make use of the particular
  form of the weight sequence and can therefore be carried through for maps
  corresponding to other critical or sub-critical Boltzmann weights. For
  the choice $g_k=(1/12)\delta_{4}(k)$, we rediscover the construction of
  the $\UIHPQ$ as described in Section~\ref{sec:BDG-mapping}. Note that for
  bipartite maps, we have $S=0$, i.e., there are no half-mobiles.
\end{remark}

We may now define maximal (and minimal) geodesics in the $\UIHPT$. Note
that vertices of the $\UIHPT$ correspond to real vertices of the
encoding. Analogously to the $\UIHPQ$, the maximal geodesic started at vertex $v$
is given by the infinite chain of vertices which are incident to the
iterated successors of the leftmost real corner $c$ belonging to $v$.

\begin{figure}[ht]
	\begin{center}
	\includegraphics[scale=0.7]{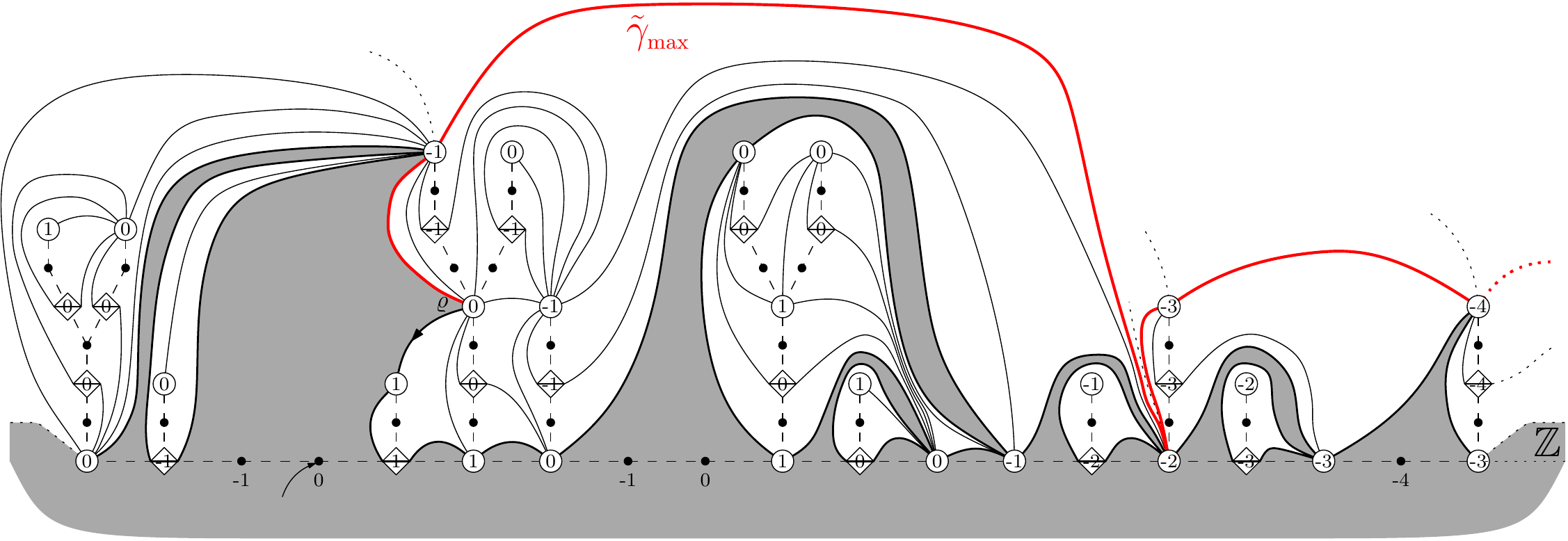}
	\end{center}
	\caption{The construction of the $\UIHPT$ from an infinite
          mobile bridge, with its maximal
          geodesic $\tgmax$.}
	\label{fig:BijTriang}
\end{figure}

Similarly, by starting from the rightmost corner, we define the minimal
geodesic emanating from $v$, and we write $\tgmax$ (or $\tgmin$) for the
maximal (or minimal) geodesic starting from the root vertex.  Moreover, we
let $\tcR_+$ and $\tcR_-$ (or $\tcR_+^{\textup{\tiny min}}$ and
$\tcR_{-}^{\textup{\tiny min}}$) denote the set of intersection times of
$\tgmax$ (or $\tgmin$) with the right and left part of the boundary,
respectively.

For characterizing $\tcR_+$ and $\tcR_{-}$ as regenerative
sets, we may argue as in the case of the $\UIHPQ$. For $j\in\Z_+$, let
\begin{align*}
\tDelta_j &= \max_{i\in\DS(\tbr_\infty)\cup\LS(\tbr_\infty)\cap[H_j,H_{j+1})}
-\left(\min_{u\in V(\tT_\infty(i))}\ell_i(u)+j\right).
\end{align*}
Here, $H_j=H_j(\tbr_\infty)$, and in hopefully obvious notation,
$\tT_\infty(i)$ is the mobile (in the case $i\in\DS(\tbr_\infty)$) or
half-mobile (in the case $i\in\LS(\tbr_\infty)$) grafted to the vertex $i$, and
$\ell_i(u)$ for $u\in V(\tT_\infty(i))$ represents its label. By replacing
$H_j$ with $H_j'$, we define $\tDelta'_j$ in a similar fashion.

\begin{prop}
\label{prop:Delta_0-UIHPT}
We have for $m\in\N$,
$$
\P\left(\tDelta_0\geq m\right)=\frac{1}{m+1},\quad\hbox{and}\quad
\P\left(\tDelta'_0\geq m\right)=\frac{1}{m+2}.
$$
\end{prop}
\begin{proof}
  We first look at $\tDelta_0$. Let $m\in\N$. Put
  $\tg(m)=1-\P\left(\tDelta_0\geq m\right)$. The arch
  decomposition corresponding to~\eqref{eq:archdecomp} reads
\begin{align*}
  \tg(m)&=\sum_{k=0}^\infty\left(\sum_{k'=0}^\infty\left(\frac{S_m}{C}\right)^{k'}\frac{\sqrt{R}}{C}\tg(m+1)\right)^k\sum_{\ell=0}^\infty\left(\frac{S_m}{C}\right)^\ell\frac{\sqrt{R}}{C}\frac{R_m}{R}\\
  &=\frac{1}{1-\left(\frac{1}{1-\frac{S_m}{C}}\frac{\sqrt{R}}{C}\right)\tg(m+1)}\frac{1}{1-\frac{S_m}{c}}\frac{\sqrt{R}}{C}\frac{R_m}{R}\,;
\end{align*}
see Figure~\ref{fig:HintProofTri}.
\begin{figure}[ht]
	\begin{center}
		\includegraphics[scale=0.9]{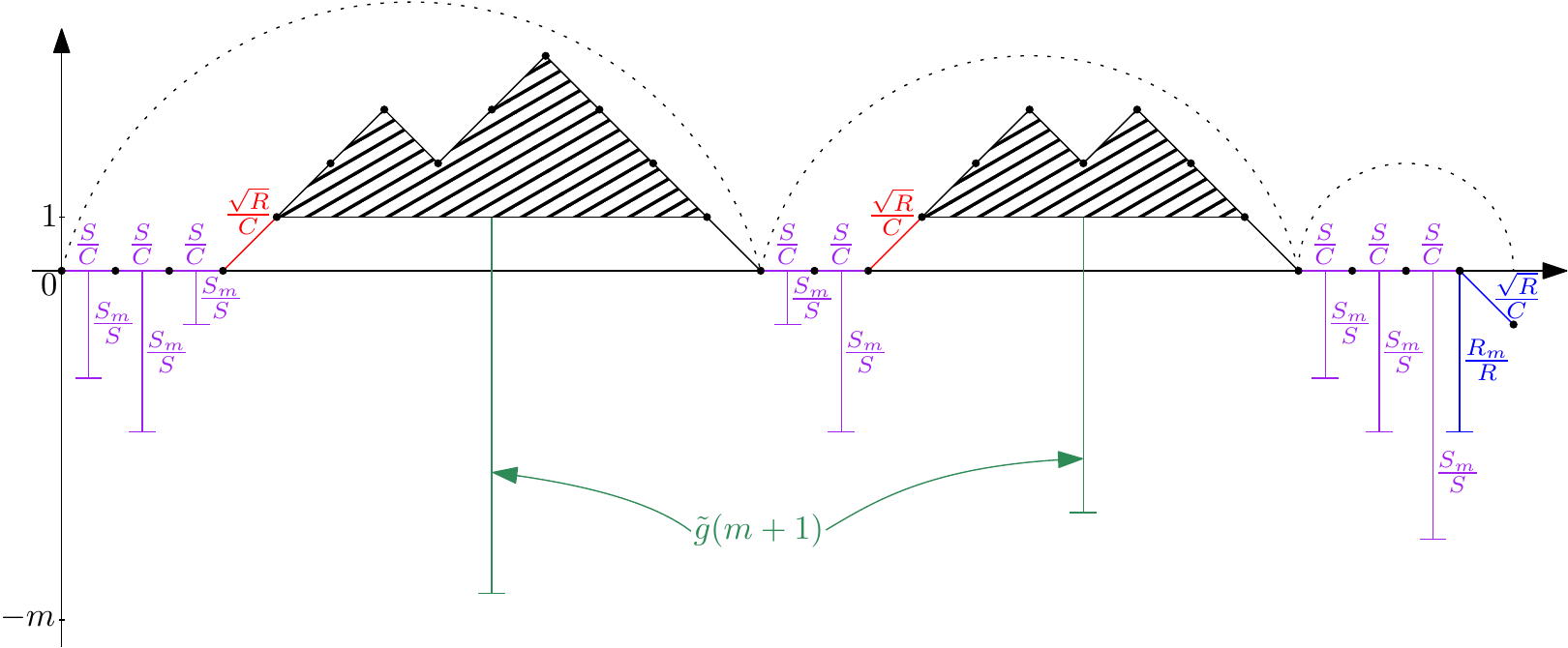}
	\end{center}
	\caption{The decomposition of the probability $\tg(m)$.}
	\label{fig:HintProofTri}
\end{figure}
The formula for $\tg$ is equivalent to
\begin{equation}
\label{eq:expr-tg}
\tg(m)\left(\frac{C}{\sqrt{R}}-\frac{S_m}{S}\frac{S}{\sqrt{R}}-\tg(m+1)\right)=\frac{R_m}{R}.
\end{equation}
We note along the way that the last expression is universal, in the sense
that it does not depend on the particular choice of the Boltzmann weights
$(g_k)_{k\in\N}$.

Back to the triangular case, by letting $y\uparrow 1$ in $(6.8)$
of~\cite{BoGu}, which corresponds to the choice of $g_3=\gcrit$, we obtain
the relations
$$
\frac{R_m}{R}=\frac{m(m+2)}{(m+1)^2},\quad
\frac{S_m}{S}=1-\frac{g_3R^2}{S}\frac{2}{(m+1)(m+2)},\quad m\in\N.
$$
Since $C=2\sqrt{R}+S$ and $2\gcrit R^{3/2}=1$, see
$(6.7)$ of~\cite{BoGu}, Equation~\eqref{eq:expr-tg} turns into
$$
\tg(m)\left(2+\frac{1}{(m+1)(m+2)}-\tg(m+1)\right)=\frac{m(m+2)}{(m+1)^2},
$$
or, with $\tf(m)=1-\tg(m)$,
\begin{equation}
\label{eq:nlequation-UIHPT}
\tf(m)-\tf(m+1)+\tf(m)\tf(m+1)+\frac{\tf(m)-1}{(m+1)(m+2)}=\frac{1}{(m+1)^2}.
\end{equation}
Of course, the last display resembles very much
Equation~\eqref{eq:nlequation} for $f$, and in fact, $\tf(m)=1/(m+1)$ is
also a solution
of~\eqref{eq:nlequation-UIHPT}. Rewriting~\eqref{eq:nlequation-UIHPT} as
$$
\tf(m+1)=\frac{(m+1)^2\tf(m)-1}{(m+1)^2(1-\tf(m))}-\frac{1}{(m+1)(m+2)},
$$
we check with the same arguments as in the proof of
Lemma~\ref{lem:nlsystem} that $\tf(m)=1/(m+1)$ is the only solution
of~\eqref{eq:nlequation-UIHPT} with $\tf(m)\in(0,1)$ for $m\in\N$,
$\tf(0)=1$ and $\lim_{m\rightarrow\infty}\tf(m)=0$. This shows
$\P(\tDelta_0\geq m)=1/(m+1)$, as claimed. The law of $\tDelta'_0$ is now
computed as in the proof of Proposition~\ref{prop:Delta_0}, using
$$\P\left(\tDelta_0<m\right)=\P\left(\tDelta'_0<m\right)\frac{R_m}{R}=\P\left(\tDelta'_0<m\right)\frac{m(m+2)}{(m+1)^2}.$$ 
\end{proof}

With the last proposition at hand, we obtain with the arguments
given in the proof of Proposition~\ref{prop:intersection} that
$$
\P\left(i\in\tcR_+\right) = \frac{1}{i+1},\quad 
\P\left(i\in\tcR_{-}\right)=\frac{2}{i+2},\quad i\in\Z_+.
$$
In particular, we again deduce that $\tgmax$ hits both parts of the
boundary in the $\UIHPT$ infinitely many times. More precisely, comparing
the last display with the analogous results obtained for $\cR_+$ and
$\cR_-$, we conclude that the intersection times of $\tgmax$ with the right
part of the boundary have exactly the same distribution as the
corresponding times of $\gmax$ in the $\UIHPQ$. On the contrary, the
maximal geodesic visits the left part of the boundary slightly more often
in the $\UIHPQ$ than in the $\UIHPT$.

A symmetry argument similar to above shows that
$\tcR_+^{\textup{\tiny min}}$ has the same law as $\tcR_-$, and we have the
inclusions $\tcR_+\subset\tcR_+^{\textup{\tiny min}}$ and
$\tcR_-^{\textup{\tiny min}}\subset \tcR_-$. Using that $\tgmax$ is proper
and hits both parts of the boundary infinitely many times, we deduce from
arguments very close to those in the proof of
Corollary~\ref{cor:propergeod} that almost surely, all geodesic rays in the
$\UIHPT$ are proper. Finally, adapting the arguments leading to
Theorem~\ref{thm:geod4} and Proposition~\ref{prop:geod4sparse}, we arrive
at the following theorem, whose details of proof we leave to the reader.
We write $\tilde{Q}_\infty^\infty$ for the $\UIHPT$ constructed in terms of
an uniform infinite mobile bridge $(\tbr_\infty,\tT_\infty)$.
\begin{thm}
\label{thm:geod3}
On a set of full probability, the following holds in the $\UIHPT$
$\tilde{Q}_\infty^\infty$: Every geodesic ray hits the boundary of the
$\UIHPT$ infinitely many times. Moreover, if $\gamma=(\gamma(i):i\in\Z_+)$
is a geodesic ray emanating from the root vertex, we have the inclusions
$$
\tcR_+\cup\tcR_-^{\textup{\tiny min}} \subseteq
\{i\in\Z_+:\gamma(i)\in\partial \tilde{Q}_\infty^\infty\}\subseteq
\tcR_+^{\textup{\tiny min}}\cup\tcR_-.
$$
The random sets $\tcR_+$ and $\tcR_-^{\textup{\tiny min}}$ (as well as
$\tcR_+^{\textup{\tiny min}}$ and $\tcR_-$) have the
same law. The distance $\delta$ between two consecutive times in $\tcR_+$
(or $\tcR_-^{\textup{\tiny min}}$) exhibits the tail behavior
$\P(\delta>m)\sim 1/\ln m$ as $m\rightarrow\infty$, whereas the distance
$\delta'$ between two consecutive times in $\tcR_+^{\textup{\tiny min}}$
(or $\tcR_-$) satisfies $\P(\delta'>m)\sim 1/(2\ln m)$.
\end{thm}

\noindent{\bf Concluding remarks.} Angel constructed in~\cite{An} the
uniform infinite triangulation with an infinite {\it simple} boundary, and
we expect that Theorem~\ref{thm:geod3} can be transferred to the model of
Angel by a pruning procedure, as in the case of the
$\UIHPQ^{(s)}$. Moreover, since the above construction of the $\UIHPT$ (or
the $\UIHPQ$) can be extended to general limits of critical or sub-critical
Boltzmann maps, the same methods can in principle be applied to study the
intersection of geodesic rays with the boundary for the full class of
models obtained in this way.

However, as it should be clear from Remark~\ref{rem:intersection-spine},
intersection properties of geodesics as studied in this paper are delicate,
and our approach requires exact calculations (or at least non-asymptotic
bounds). In the pure quadrangular and triangular cases at criticality, the
expressions for $R_m$ and $S_m$ are particularly simple, so that we can
compute the laws of $\Delta_0$ and $\tDelta_0$ explicitly. See $(5.11)$
of~\cite{BoGu} for the general form of $R_m$, which involves so-called
Hankel determinants. For a more general treatment,
Equation~\eqref{eq:expr-tg} is model-independent and may serve as a
starting point for further investigations.\newline

\noindent {\bf Acknowledgments.}
We would like to thank Nicolas Curien for asking us whether geodesic rays
in the $\UIHPQ$ intersect the boundary infinitely often. Moreover, we thank
J\'er\'emie Bouttier for a helpful discussion, and the referee for pointing
us to~\cite[Proposition 2.4]{ChDu}.
\bibliography{UIHPQ-geod}
\bibliographystyle{abbrv}
\end{document}